%% file: MousleyRussell_Morse_boundary_HHG_final.tex
\documentclass[11pt]{article}

\usepackage{graphicx}
\usepackage{transparent}
\usepackage{geometry}
\usepackage{amsmath}
\usepackage{amsfonts}
\usepackage{amsthm}
\usepackage{amssymb}
\usepackage{mathdots}
\usepackage{latexsym}
\usepackage{mathrsfs}
\usepackage{subfig}
\usepackage{enumerate}
\usepackage{rotating}
\usepackage{multirow}
\usepackage{hyperref}
\usepackage{pgfplots}
\usepackage{lineno, blindtext}
\usepackage{cancel}
\usepackage{wrapfig}
\usepackage{float}
\usepackage{multicol} 
\usepackage{graphicx}
\usepackage{dsfont}
\usepackage{yfonts}
\usepackage{array} 
\usepackage[noadjust]{cite}

\newtheorem{theorem}{Theorem}[section]
\newtheorem{lemma}[theorem]{Lemma}
\newtheorem{proposition}[theorem]{Proposition}
\newtheorem{corollary}[theorem]{Corollary}

\newtheorem*{main}{Main Theorem}

\theoremstyle{definition}
\newtheorem{example}[theorem]{Example}
\newtheorem{remark}[theorem]{Remark}
\newtheorem{definition}[theorem]{Definition}
\newtheorem{convention}[theorem]{Convention}

\newcommand{\ol}{\overline}
\newcommand{\T}{\mathcal{T}}
\newcommand{\diam}{{\rm diam}}

\newcommand{\N}{\mathcal{N}}

\newcommand{\M}{\mathcal{M}} 
\newcommand{\R}{\mathbb{R}}

\newcommand\mf[1]{\mathfrak{#1}}
\newcommand\mc[1]{\mathcal{#1}}

\usepackage{lmodern}
\usepackage{tikz}
\usetikzlibrary{decorations.pathmorphing}

\tikzset{snake it/.style={decorate, decoration=snake}}

\usetikzlibrary{positioning,arrows.meta}
\usetikzlibrary{decorations.markings}
\usetikzlibrary{calc}

\tikzset{
  fermion/.style={draw=black, postaction={decorate},decoration={markings,mark=at position .55 with {\arrow{Latex}}}},
  vertex/.style={draw,shape=circle,fill=black,minimum size=2pt,inner sep=0pt},
  photon/.style={wavy semicircle,wave amplitude=0.2mm,wave count=6}
}

\newif\ifmirrorsemicircle
\tikzset{
    wave amplitude/.initial=0.2cm,
    wave count/.initial=8,
    mirror semicircle/.is if=mirrorsemicircle,
    mirror semicircle=false,
    wavy semicircle/.style={
        to path={
            let \p1 = (\tikztostart),
            \p2 = (\tikztotarget),
            \n1 = {veclen(\y2-\y1,\x2-\x1)},
            \n2 = {atan2(\x2-\x1,\y2-\y1))} in
            plot [
                smooth,
                samples=(\pgfkeysvalueof{/tikz/wave count}+0.5)*8+1, 
                domain=0:1,
                shift={($(\p1)!0.5!(\p2)$)}
            ] ({ 
                (\x*180-\n2 + 180 + \ifmirrorsemicircle 1 \else -1 \fi * 90%
            }:{ 
                (%
                    \n1/2+\pgfkeysvalueof{/tikz/wave amplitude} * %
                    sin(
                        \x * 360 * (\pgfkeysvalueof{/tikz/wave count} + 0.5%
                    )%
                )%
            })
        } (\tikztotarget)
    }
}

\newcommand{\hgline}[2]{
\pgfmathsetmacro{\thetaone}{#1}
\pgfmathsetmacro{\thetatwo}{#2}
\pgfmathsetmacro{\theta}{(\thetaone+\thetatwo)/2}
\pgfmathsetmacro{\phi}{abs(\thetaone-\thetatwo)/2}
\pgfmathsetmacro{\close}{less(abs(\phi-90),0.0001)}
\ifdim \close pt = 1pt
    \draw[black] (\theta+180:1) -- (\theta:1);
\else
    \pgfmathsetmacro{\R}{tan(\phi)}
    \pgfmathsetmacro{\distance}{sqrt(1+\R^2)}
    \draw[black, thick] (\theta:\distance) circle (\R);
\fi
}

\newcommand{\hglinewhite}[2]{
\pgfmathsetmacro{\thetaone}{#1}
\pgfmathsetmacro{\thetatwo}{#2}
\pgfmathsetmacro{\theta}{(\thetaone+\thetatwo)/2}
\pgfmathsetmacro{\phi}{abs(\thetaone-\thetatwo)/2}
\pgfmathsetmacro{\close}{less(abs(\phi-90),0.0001)}
\ifdim \close pt = 1pt
    \draw[black] (\theta+180:1) -- (\theta:1);
\else
    \pgfmathsetmacro{\R}{tan(\phi)}
    \pgfmathsetmacro{\distance}{sqrt(1+\R^2)}
    \fill[white] (\theta:\distance) circle (\R);
\fi
}

\begin{document}
\pagenumbering{arabic}

\title{Hierarchically hyperbolic groups are determined \\ by their Morse boundaries}
\date{}
\author{Sarah C. Mousley and Jacob Russell}

\maketitle 
\abstract{We generalize a result of Paulin on the Gromov boundary of hyperbolic groups to the Morse boundary of proper, maximal hierarchically hyperbolic spaces admitting cocompact group actions by isometries.
Namely we show that if the Morse boundaries of two such spaces each contain at least three points, then the spaces are quasi-isometric if and only if there exists a 2--stable, quasi-m\"obius homeomorphism between their Morse boundaries. Our result extends a recent result of Charney--Murray, who prove such a classification for  CAT(0) groups, and is new for mapping class groups and the fundamental groups of $3$--manifolds without Nil or Sol components.}

\section{Introduction} 

The Gromov boundary has proved a profitable tool in the study of the coarse geometry of hyperbolic groups. Every quasi-isometry between hyperbolic groups extends to a homeomorphism between the Gromov boundaries of the groups, and so it is natural to ask if the homeomorphism type of the boundary determines the quasi-isometry type of the group. In general, this is not the case as examples of non-quasi-isometric groups with homeomorphic Gromov boundary are given in \cite{counter1}, \cite{counter2}, and \cite{counter3}. However,  Paulin \cite{paulin} defined a cross-ratio on the Gromov boundary and proved that if the boundary homeomorphism is also quasi-m\"obius with respect to this cross-ratio, then it is necessarily induced by a quasi-isometry between the groups. 

In an effort to generalize the Gromov boundary to a broader class of spaces, Cordes \cite{cordes}  introduced a boundary for proper geodesic metric spaces called the Morse boundary.  The Morse boundary captures the asymptotics of the ``hyperbolic-like" directions in the space and agrees with the Gromov boundary when the space is hyperbolic. Cordes proved that the Morse boundary is a quasi-isometry invariant in the same sense as the Gromov boundary---any quasi-isometry between proper geodesic metric spaces induces a homeomorphism of their Morse boundaries (see Proposition \ref{prop:cordes_inducedmap}).  In light of this, Charney and Murray \cite{CM} established a Morse boundary version of Paulin's theorem for CAT(0) groups. The main result of our paper provides an analogous result for hierarchically hyperbolic groups, a class containing mapping class groups of finite type surfaces and right-angled Artin groups.

\begin{main} \label{thrm:main} 
Suppose $G$ and $H$ are hierarchically hyperbolic groups whose Morse boundaries have at least three points. A homeomorphism between the Morse boundaries of $G$ and $H$ is induced by a quasi-isometry if and only if it and its inverse are  2--stable and quasi-m\"obius.
\end{main}

In fact, we will show that the above theorem holds for any proper geodesic metric spaces admitting cocompact group actions by isometries and satisfying the  \emph{small cross-ratio property}, a technical condition introduced in Section \ref{subsec:controlled_distortion}.  This allows us to extend the Main Theorem to include the fundamental groups of  $3$--manifolds without Nil or Sol components.

Paulin's proof hinged upon the observation that to each triple of distinct points $a,b,$ and $c$ in the Gromov boundary of a hyperbolic space, there is an associated uniformly bounded diameter set of points, thought of as the coarse center of an ideal triangle with vertices $a,b,$ and $c$. Thus given a homeomorphism between the boundaries of two hyperbolic spaces,  a map can be built between the spaces by sending centers to centers,  and Paulin shows the quasi-m\"obius condition is sufficient to ensure this map is a quasi-isometry.  The approach in \cite{CM} is very similar in spirit and substance to that of Paulin.  Our approach follows the overall strategy of \cite{CM} very closely; however, Charney and Murray are able to take advantage of the fact that in CAT(0) spaces, Morse geodesics are strongly contracting. This is the point where our paper diverges. In general, Morse geodesics are not strongly contracting \cite{contracting}, 
and the bulk of our paper is devoted to overcoming this obstacle. An advantage of our work is that all of our preliminary results hold for any finitely generated group. So if one could demonstrate that every finitely generated group has the small cross-ratio property, then one could immediately extend our Main Theorem to all finitely generated groups (see Theorems \ref{thrm:quasi-isometry_to_quasi-mobius}, \ref{thrm:main_body_part1}, and \ref{thrm:main_body_part2}). 

In Section 2 we will collect the required background we will need about coarse geometry, the Morse boundary, and hierarchically hyperbolic spaces. In Section 3 we will uncover some basic facts about infinite Morse geodesics and ideal Morse triangles, which are reminiscent of the properties of infinite geodesics and ideal triangles in a hyperbolic space.  This will allow us to both define a cross-ratio on the Morse boundary and to associate sets of bounded diameter to triples of distinct boundary points.  In Section 4 we will show 
that every quasi-isometry between proper geodesic metric spaces induces a 2--stable, quasi-m\"obius homeomorphism of  Morse boundaries (Theorem \ref{thrm:quasi-isometry_to_quasi-mobius}) and that the converse is also true for spaces that are \emph{Morse centered} and have the small cross-ratio property (Theorems \ref{thrm:main_body_part1} and \ref{thrm:main_body_part2}). In Section \ref{section:HHSs} we prove that many hierarchically hyperbolic spaces, including all hierarchically hyperbolic groups, have the small cross-ratio property (Proposition \ref{small:cross-ratio}), thus proving the Main Theorem from the introduction. 

\

{\bf Acknowledgments:} The authors would like to thank Devin Murray, Ruth Charney, and  Matthew Cordes for helpful conversation about the Morse boundary and Mark Hagen for his assistance with hierarchically hyperbolic spaces. 
The authors are also very grateful for the support and guidance of their advisors, Chris Leininger and Jason Behrstock. The second author would like to thank Chris Leininger for his mentorship and hospitality during a visit funded by NSF grants DMS 1107452, 1107263, 1107367 ``RNMS: Geometric Structures and Representation Varieties" (the GEAR Network), where the work on this project began.

\section{Background and definitions} 

\subsection{Coarse geometry}

Throughout this paper $X$ and $Y$ will denote geodesic metric spaces, with distance functions $d_X$ and $d_Y$ respectively. When $X$ is understood, we write $d$ instead of $d_X$. Additionally, given $\lambda \geq 1$ and $\varepsilon \geq 0$, we will let $A \stackrel{\lambda,\varepsilon}{\asymp} B$ denote $\frac{1}{\lambda} A - \varepsilon \leq B \leq \lambda A + \varepsilon$.

We are primarily interested in studying the large scale or coarse geometry of geodesic metric spaces. We say a function $f\colon X \rightarrow Y$ is a \emph{$(\lambda,\varepsilon)-$quasi-isometric embedding} if  

\[  d_X(x,y)  \stackrel{\lambda,\varepsilon}{\asymp} d_Y(f(x),f(y)) \hspace{7pt} \text{ for all } x,y \in X.\] 
 If there  exists a map $\hat{f}\colon Y \rightarrow X$ and a constant $D\geq0$ such that $d_X(\hat{f} \circ f (x), x) \leq D$ for all $x \in X$ and $d_Y(f \circ \hat{f} (y), y) \leq D$ for all $y \in Y$, then we say that $\hat{f}$ is a \emph{quasi-inverse} of $f$.  If $f$ is a quasi-isometric embedding that has a quasi-inverse, then we say $f$ is a \emph{quasi-isometry} and that $(X,d_X)$ and $(Y,d_Y)$  are \emph{quasi-isometric}.

From the point of view of quasi-isometries, there is little distinction between a point and a set of uniformly bounded diameter.  Thus it is often convenient to utilize coarse maps. A \emph{coarse map} is a function $f\colon X \rightarrow 2^Y$ such that $f(x) \neq \varnothing$ and the diameter of $f(x)$ is uniformly bounded for all $x \in X$. We will let $f^{-1}(y)$ denote the set $\{x \in X : y \in f(x)\}$. A coarse map $f\colon X \rightarrow 2^Y$  induces a (non-canonical) function $\overline{f}\colon X \rightarrow Y$, by defining $\overline{f}(x)$ to be any choice of point in $f(x)$. 
 
 A $(\lambda, \varepsilon)$--\emph{quasi-geodesic} $\alpha$ in $X$ is a $(\lambda, \varepsilon)$--quasi-isometric embedding of a closed interval $I \subseteq \mathbb{R}$ into $X$.  If $I$ is compact, we say $\alpha$ is a \emph{finite} quasi-geodesic; otherwise, $\alpha$ is an \emph{infinite} quasi-geodesic. If $I$ has only one finite endpoint, then we say $\alpha$ is a \emph{quasi-geodesic ray}.

\subsection{Gromov boundary of hyperbolic spaces}

A central and fruitful class of spaces studied in coarse geometry are the hyperbolic spaces introduced by Gromov \cite{gromov1,gromov2}.

Let $X$ be a geodesic metric space and $\delta \geq 0$. A geodesic triangle in $X$ is \emph{$\delta$--slim} if each side is contained in the $\delta$--neighborhood of the union of the of the other two sides. We say $X$ is \emph{$\delta$--hyperbolic} if every triangle in $X$ is $\delta$--slim. 

If $X$ is $\delta$--hyperbolic, then there exists a quasi-isometry invariant boundary of $X$, called the \emph{Gromov boundary of $X$}, denoted by $\partial_G X$. To define $\partial_G X$, we need the notion of Gromov product. For  $x,y, p \in X$,  the \textit{Gromov product of $x$ and $y$ at $p$} is
\[(x,y)_p = \frac{1}{2}(d(x,p) + d(y,p) - d(x,y)).\]

We now define $\partial_G X$ to be the collection of equivalence classes of sequences $(x_n)$ in $X$ with the property that $\liminf \limits_{n,m\rightarrow \infty} (x_n, x_m)_p=\infty$ for some (any) $p \in X$, where two such sequences $(x_n)$ and $(y_m)$ are equivalent if $\liminf \limits_{n,m\rightarrow \infty } (x_n,y_m)_p =\infty$ for some (any) $p \in X$. 

\subsection{Morse boundary} 
In \cite{cordes}, Cordes introduced the Morse boundary for proper geodesic metric spaces. Cordes' work generalized both the Gromov boundary of hyperbolic spaces and the contracting boundary of CAT(0) spaces introduced by Charney and Sultan \cite{CS}. We will now recall Cordes' construction of the Morse  boundary and several of the results in \cite{cordes}.

Let $N\colon [1,\infty)  \times [0,\infty) \rightarrow [0,\infty)$ be a function. 
A quasi-geodesic $\gamma$ is \textit{$N$--Morse} if every $(\lambda, \varepsilon)$--quasi-geodesic with endpoints on $\gamma$ is contained in the $N(\lambda, \varepsilon)$--neighborhood of $\gamma$. We call $N$ a \textit{Morse gauge} for $\gamma$.

Let $X$ be a proper geodesic metric space. We say that two quasi-geodesic rays $\alpha$ and $\beta$ in $X$ are \textit{asymptotic} if the Hausdorff distance between the images of $\alpha$ and $\beta$, denoted by $d_{Haus}(\alpha, \beta)$, is finite. Choose a basepoint $p\in X$ and define $\partial_v X_p$ to be the set of equivalence classes of geodesic rays based at $p$, where two rays are equivalent if they are asymptotic. 
We topologize $\partial_v X_p$ by declaring  a subset $B$ of $\partial_vX_p$ to be closed if and only if the following statement is true. Suppose $(b_n)$ is a sequence in $B$, $(\beta_n)$ is a sequence of rays satisfying $[\beta_n]=b_n$, and $b \in \partial_vX_p$. If every subsequence of $(\beta_n)$ contains a subsequence that converges uniformly on compact sets to a ray in $b$, then $b$ is in $B$. 

Define $\M$ to be the collection of all Morse gauges for geodesic rays based at $p$. For each $N\in \M$, define 
\[\partial^N X_p=\{[\alpha]\in \partial_v X_p: \alpha \text{ is equivalent to an } N\text{--Morse geodesic ray based at } p \}.\]
Observe that because $\partial^N X_p$ is a subset of $\partial_vX_p$, we can equip it with the subspace topology, which we denote by $\mathcal{T}^N_v$. Cordes formulated a topology for  $\partial^NX_p$ slightly differently; however,  using Lemma \ref{Morse:end-points} and the  following lemma, it is a short exercise to verify that the two topologies are indeed the same.  (Our description of the topology allows us to prove Lemma \ref{lemma:convergence_in_strata} with ease.)

\begin{lemma}[\cite{cordes} Lemma 2.10]  \label{lemma:cordes_uniform_convergence} 
Let $X$ be a geodesic space, and let $\{\gamma_i\colon[0, \infty) \rightarrow X\}$ be a sequence of $N$--Morse geodesic rays that converge uniformly on compact sets to a geodesic ray $\gamma$. Then $\gamma$ is $N$--Morse. 
\end{lemma}

There is a natural partial order on $\M$.  We define $N \leq N'$ if  $N(\lambda,\varepsilon)\leq N'(\lambda,\varepsilon)$ for all $(\lambda, \varepsilon) \in [1,\infty) \times [0,\infty)$. Observe that for $N \leq N'$, the inclusion $\partial^N X_p \rightarrow \partial^{N'} X_p$ is continuous. As in \cite{cordes}, we define the \textit{Morse boundary} of $X$ to be 
\[\partial_M X_p= {\varinjlim_{\M}} \  \partial^N X_p,\]
equipped with the direct limit topology. That is, $\partial_M X_p=\bigcup\limits_{N\in\M}\partial^N  X_p$ and is given the finest topology so that for each Morse gauge $N$, the inclusion $\partial^N X_p \rightarrow \partial_M X_p$ is continuous. 

The subspace topology on $\partial^N X_p$ as a subset of $\partial_M X_p$ may be coarser than $\mathcal{T}^N_v$. Thus, given a sequence $(a_n)$ in $\partial^N X_p$ and $a \in \partial_M X_p$ such that $a_n \rightarrow a$  in  $\partial_MX_p$, it is not immediately clear that $a_n \rightarrow a$ in $(\partial^N X_p, \mathcal{T}^N_v)$. However, this turns out to be true.

\begin{lemma} \label{lemma:convergence_in_strata} 
Let $(a_n)$ be a sequence  converging to $a$ in $\partial_M X_p$. Suppose there exists a gauge $N$ such that each $a_n \in \partial^N X_p$. Then $a_n \rightarrow a$ in $(\partial^N X_p, \mathcal{T}_v^N)$.
\end{lemma}

\begin{proof} 
The universal property of direct limits implies that the inclusion $\partial_M X_p \rightarrow \partial_v X_p$  is continuous. Assume $a_n \rightarrow a$ in $\partial_M X_p$. Continuity implies that $a_n \rightarrow a$ in $\partial_v X_p$. Lemmas \ref{lemma:cordes_uniform_convergence} and \ref{Morse:end-points} imply that $\partial^NX_p$ is a closed subset of $\partial_vX_p$.  Thus, it must be that $a \in \partial^NX_p$. So, by  definition of  the subspace topology,  $a_n \rightarrow a$ in $(\partial^N X_p, \mathcal{T}_v^N)$ as desired. 
\end{proof} 

We now recall several propositions of Cordes, the first of which will be useful for determining convergence in $(\partial^N X_p, \mathcal{T}_v^N)$. 

\begin{proposition}[\cite{cordes} Lemma 3.1] \label{prop:fundneigh} There exists an increasing function $\eta: \M \rightarrow [0,\infty)$ with unbounded image such that for every $N$--Morse geodesic ray $\gamma$ based at $p$, the collection  $\{V_{n}(\gamma)\}$ is a fundamental system of neighborhoods of $[\gamma]$ in $(\partial^N X_p, \T_v^N)$, where   
\[V_{n}(\gamma)=\{[\alpha] \in \partial^N X_p: d(\alpha(t),\gamma(t)) \leq \eta(N) \text{ for all } t \leq n \} .\] 
In particular, the space $(\partial^N X_p, \mathcal{T}^N_v)$ is Hausdorff for each Morse gauge $N$. 
\end{proposition} 

For any pair of points $a,b \in \partial_M X_p$,  a bi-infinite quasi-geodesic $\gamma$ in $X$  \emph{from $a$ to $b$} is one with the property that $\gamma$ restricted to $(-\infty, 0]$ is asymptotic to a representative of $a$ and $\gamma$ restricted to $[0,\infty)$ is asymptotic to a representative of $b$. 

\begin{proposition}[\cite{cordes} Proposition 3.11]\label{prop:cordes_visual}
Let $X$ be a proper geodesic space and $a,b \in \partial^N X_p$.  There exists a bi-infinite Morse geodesic from $a$ to $b$, whose gauge depends only on $N$. 
\end{proposition}

\begin{proposition}[\cite{cordes} Proposition 3.7] \label{prop:cordes_inducedmap}
Let $f: X \rightarrow Y$ be a $(\lambda, \varepsilon)$--quasi-isometry of proper geodesic spaces. Then for each $N$--Morse geodesic ray $\gamma$ based at $p\in X$, $f(\gamma)$ stays bounded distance from an $N'$--Morse geodesic ray based at $f(p)$, say $\ol{f(\gamma)}$, where $N'$ depends only on $N, \lambda$, and $\varepsilon$. This allows us to build a homeomorphism $f_*: \partial_M X_{p} \rightarrow \partial_M Y_{f(p)}$ by defining $f_*([\gamma]) =[\ol{f(\gamma)}]$. 
\end{proposition}

Proposition \ref{prop:cordes_inducedmap} shows that changing the basepoint $p$ to $p'$ produces a homeomorphic boundary via a canonical homeomorphism induced by the quasi-isometry sending $p$ to $p'$ and fixing all other points. Thus, throughout the remainder of this paper,  $\partial X$ will denote the Morse boundary $\partial_M X_p$.  If $X$ is hyperbolic, then there exists a Morse gauge $N$ such that every geodesic is $N$--Morse.  It is then readily seen that the Gromov and Morse boundaries are homeomorphic, and thus we will also use $\partial X$ to denote $\partial_G X$.

\subsection{Hierarchically hyperbolic spaces and groups} 
Introduced by Behrstock, Hagen, and Sisto in \cite{BHSI} and revised in \cite{BHSII}, hierarchically hyperbolic spaces generalize the Masur--Minsky hierarchy machinery  for the mapping class group \cite{MMI,MMII}. The definition of a hierarchically hyperbolic space is quite involved and since we will not need the full breath of the definition, we  recall only the relevant facts and features we will need for Section \ref{sec:HHS_smallcross}.  We direct the curious reader to \cite{BHSII} for the complete definition and to \cite{HHS_survey} for a survey of the current state of the theory.

Let $\mc{X}$ be a geodesic metric space. 
A hierarchically hyperbolic space (HHS) structure on $\mc{X}$ begins with a partially ordered index set $\mf{S}$ with a unique maximal element, which we will always denote by $S$. For each $U \in \mf{S}$, there exists a hyperbolic space $CU$ and a uniformly coarsely Lipschitz projection map $\pi_U: \mc{X} \rightarrow CU$.  Additionally, many axioms govern how $\mf{S}$, $\{CU\}_{U\in\mf{S}}$, and $\{\pi_U\}_{U\in\mf{S}}$ are related. When $\mc{X}$ is equipped with an HHS structure, we call $\mc{X}$ a \textit{hierarchically hyperbolic space} (HHS), and  we often write $(\mc{X},\mf{S})$ to indicate that $\mc{X}$ is equipped with a particular HHS structure with index set $\mf{S}$.

A finitely generated group $G$ is a \textit{hierarchically hyperbolic group} (HHG) if $G$ admits a proper, cobounded, uniform quasi-action on some HHS $(\mc{X},\mf{S})$, and $G$ acts on $\mf{S}$ with finitely many orbits by order preserving automorphisms. Additionally, the quasi-action must preserve some aspects of the HHS structure. If $G$ is an HHG, by virtue of its quasi-action on $(\mc{X},\mf{S})$,  there exists a quasi-isometry $f \colon G \rightarrow \mc{X}$, where $G$ is equipped with the word metric with respect to any finite generating set. We can thus put an HHS structure on $G$ using the index set $\mf{S}$, hyperbolic spaces $\{CU\}_{U\in\mf{S}}$, and projection maps $\{\pi_U \circ f\}_{U\in \mf{S}}$.

\begin{example}\label{ex:HHS_examples}
The following  admit HHS structures: mapping class groups and Teichm\"uller spaces (with either the Weil-Petersson or Teichm\"uller metric) of finite type surfaces \cite{MMI}, \cite{MMII}, \cite{Brock}, \cite{Behrstock}, \cite{Rafi}, \cite{BKMM}, \cite{Durham}, \cite{EMR}, CAT(0) cubical groups \cite{HS}, right-angled Artin and Coxeter groups \cite{BHSI}, fundamental groups of 3-manifolds without Nil or Sol components in their prime decomposition \cite{BHSII}, and the non-separating curve graph \cite{vokes}.  Further, mapping class groups and CAT(0) cubical groups are HHGs. 
\end{example}

In \cite{ABD}, Abbott, Behrstock, and Durham introduced \textit{almost HHS structures}, a slight weakening of the HHS axioms. Almost HHS structures maintain nearly all of the features of an HHS structure including the partially ordered index set $\mf{S}$ with a unique maximal element $S$ and the associated hyperbolic spaces and projection maps for each $U \in \mf{S}$. 

The crux of HHS theory is that the coarse geometry of $\mc{X}$ can be recovered from the geometry of the associated spaces (whose geometry is in turn elucidated by hyperbolicity). This philosophy is displayed most predominately by the following hallmark theorem.

\begin{theorem}[\cite{BHSII} Theorem 4.5, \cite{ABD} Thoerem 7.2 ]\label{dist:formula}
Let $(\mc{X}, \mf{S})$ be an almost HHS.  For each $\sigma$ sufficiently large, there exist $A\geq 1$ and $B\geq 0$ such that 
\[d_{\mc{X}} (x,y) \stackrel{A,B}{\asymp} \sum\limits_{U\in\mf{S}} [d_{CU}( \pi_U(x), \pi_U(y) )]_{\sigma}  \hspace{20pt} \text{for all} \hspace{7pt}  x,y \in \mc{X},\]
where $[ M ]_\sigma = M$ if $M\geq \sigma$ and $0$ otherwise.
\end{theorem}

\subsection{Morse quasi-geodesics in hierarchically hyperbolic spaces} 

In \cite{ABD}, for HHSs satisfying the bounded domain dichotomy, a characterization of Morse quasi-geodesics is given in terms of their projections to the associated hyperbolic spaces.  The bounded domain dichotomy, defined in \cite{ABD},  is a technical condition satisfied by many HHSs including those from Example \ref{ex:HHS_examples}. As pointed out in \cite{ABD}, all HHGs satisfy the bounded domain dichotomy.

\begin{theorem}[\cite{ABD} Theorem 6.2, 7.2]\label{thrm:HHS_Stability}
Let $\mc{X}$ be a geodesic metric space that admits an HHS structure with the bounded domain dichotomy.  Then there exists an almost HHS structure $\mf{S}$ for $\mc{X}$ such that for every quasi-geodesic $\gamma$ in $\mc{X}$ the following are equivalent: 
\begin{enumerate}
\item $\gamma$ is $N$--Morse. 
\item There exists $B$ such that  $\diam_{CU}(\pi_U(\gamma)) \leq B$ for all $U\in\mf{S} - \{S\}$.
\item $\pi_S \circ \gamma$ is an $(L,L)$--quasi-geodesic. 
\end{enumerate} Further $N, B$, and $L$ determine each other.
\end{theorem}

We shall call any almost HHS where conditions (1)--(3) of Theorem \ref{thrm:HHS_Stability} are equivalent a \emph{maximal almost HHS}. Observe that if $(\mc{X}, \mf{S})$ is a maximal almost HHS, we can build an injective map 
\[ \ol{\pi}_S: \partial \mc{X} \rightarrow \partial CS \hspace{10pt} \text{via} \hspace{10pt} [\gamma] \mapsto (\pi_S(\gamma(n)))_{n \in \mathbb{N}}.\]

\section{Preliminaries} 

\subsection{Properties of Morse geodesics}

Throughout this section, $X$ will denote a geodesic metric space. We begin by recalling a few standard facts about Morse geodesics.  

\begin{lemma}\label{lemma:close_implies_Morse}
If $\alpha$ is an $N$--Morse quasi-geodesic and $\beta$ is a quasi-geodesic such that \\ $d_{Haus}(\alpha,\beta) \leq D$, then $\beta$ is $N'$--Morse, where $N'$ depends only on $N$ and $D$.
\end{lemma}

\begin{lemma}\label{lemma:morse_subgeodesics}
For each Morse gauge $N$, there exists $N'$, depending only on $N,$ such that if $\alpha$ is an $N$--Morse geodesic in $X$ and $\beta$ is a subgeodesic of $\alpha$, then $\beta$ is $N'$--Morse.
\end{lemma}

\begin{proof}
Let $\gamma: [a,b] \rightarrow X$ be a $(\lambda, \varepsilon)$--quasi-geodesic with $\gamma(a)$ and $\gamma(b)$ on $\beta$. First we prove the case when $\gamma$ is continuous.  Let $D= N(\lambda,\varepsilon)$, and suppose $\gamma$ is not contained in the $D$--neighborhood of $\beta$. Continuity of $\gamma$ ensures that there exists $c\in [a,b]$ such that $\gamma(c)$ is within $D$ of each component of $\alpha - \beta$. If $\alpha-\beta$ has two components, this implies that the length of $\beta$ is bounded above by $2D$, and thus the length of $\gamma$ is no more than $2D\lambda+\varepsilon$, which guarantees that $\gamma \subseteq \mc{N}_{2D\lambda+\varepsilon}(\beta)$. A similar argument shows that if $\alpha-\beta$ has just one component, then $\gamma$ is contained in a uniform (in terms of $N, \lambda$, and $\varepsilon$) neighborhood of $\beta$. 
 
If $\gamma$ is not continuous, then there exists a continuous $(\lambda, 2\lambda+2\varepsilon)$--quasi-geodesic $\gamma':[a,b] \rightarrow X$ with the same endpoints as $\gamma$ satisfying $d_{Haus}(\gamma, \gamma')\leq \lambda + \varepsilon$ (see \cite{BH} III.H Lemma 1.11). Thus repeating the above argument with $\gamma'$ produces $D'$ depending only on $N$, $\lambda$, and $\varepsilon$ such that $\gamma \subseteq \mc{N}_{D'}(\beta)$. Hence there exists an $N'$ depending only on $N$ such that $\beta$ is $N'$--Morse.
\end{proof} 

\begin{convention} In light of Lemma \ref{lemma:morse_subgeodesics},  for the remainder of this paper when we say a geodesic is $N$--Morse, we will always assume that $N$ was taken large enough so that all subgeodesics  are also $N$--Morse. 
\end{convention} 

In  \cite{cordes} Section 2, Cordes outlines many properties of Morse geodesics, showing that Morse geodesics behave like geodesics in hyperbolic spaces. The following two lemmas are based on \cite{cordes} Section 2, but reformulated to fit our needs.

\begin{lemma} \label{(3,0)-lemma}
Let $\alpha$ be a $(\lambda,\varepsilon)$--quasi-geodesic in $X$ and $x\in X$. If $x'$ is a point on $\alpha$ closest to $x$ and $\gamma$ is the concatenation of a geodesic connecting $x$ to $x'$ and a subsegment of $\alpha$ with endpoint $x'$, then $\gamma$ is a $(2\lambda+1,\varepsilon)$--quasi-geodesic. 
\end{lemma}

\begin{proof}
An exercise nearly identical to the proof of Lemma 2.2 in \cite{cordes}.
\end{proof}

\begin{lemma} \label{var:cordes}
Let $\alpha$ be an $N$--Morse geodesic ray in $X$.  If $\beta$ is a $(\lambda,\varepsilon)$--quasi-geodesic ray such that $\alpha$ and $\beta$ have the same initial point and $d_{Haus}(\alpha, \beta) < \infty$, then there exists $D$, depending only on $N$, $\lambda$, and $\varepsilon$, such that $\beta \subseteq \mc{N}_{D}(\alpha)$.
\end{lemma}

\begin{proof}
Let $p$ be the shared initial point of $\alpha$ and $\beta$ and $q$ be a point on $\beta$. Since $d_{Haus}(\alpha,\beta)<\infty$, there must exist a point $x$ on $\alpha$ and $x'$ on $\beta$ such that $x'$ is a closest point to $x$ on $\beta$ and $q$ is on the subsegment of $\beta$ between $p$ and $x'$. Let $\beta'$ be the concatenation of the subsegment of $\beta$ between $p$ and $x'$ with a geodesic from $x$ to $x'$. By Lemma \ref{(3,0)-lemma}, $\beta'$ is a $(2\lambda+1, \varepsilon)$--quasi-geodesic, and thus $q \in \beta' \subseteq \mc{N}_{N(2\lambda+1,\varepsilon)}(\alpha)$.  Hence we have $\beta\subseteq \mc{N}_{N(2\lambda+1,\varepsilon)}(\alpha)$.
\end{proof}

\begin{corollary}\label{inf:to:finite}
Let $\alpha$ be an $N$--Morse geodesic ray in $X$ and $\beta:[0,\infty) \rightarrow X$ a $(\lambda,\varepsilon)$--quasi-geodesic ray asymptotic to $\alpha$. There exists  $t_0\geq0$ such that $\beta[t_0,\infty) \subseteq \N_D(\alpha)$, where $D$ depends only on $N$, $\lambda,$ and $\varepsilon$.
\end{corollary}

\begin{proof}
 Pick any point $p$ on $ \alpha$ and let $p'$ be a point on $\beta$ closest to $p$. Let $\beta'$ be the concatenation of a geodesic from $p$ to $p'$ and the infinite component of $\beta[0,\infty) -\{p'\}$.  Lemmas \ref{(3,0)-lemma} and \ref{var:cordes} provide a constant $D$ depending only on $N$, $\lambda$, and $\varepsilon$ such that $\beta' \subseteq \mc{N}_D(\alpha)$.
\end{proof}

We shall also require the following result of Cordes \cite{cordes}. 
\begin{lemma} [\cite{cordes} Lemma 2.1, 2.2, 2.3] \label{cordes:Morse:geodesics}
Let $X$ be a geodesic metric space.
\begin{enumerate}[(1)]
    \item If $\alpha$ is a finite $N$--Morse geodesic in $X$ and $\beta$ is a $(\lambda,\varepsilon)$--quasi-geodesic with the same endpoints as $\alpha$, then there exists a constant $D$ and a Morse gauge $N'$ both depending only on $N$, $\lambda$, and $\varepsilon$ such that $d_{Haus}(\alpha,\beta) \leq D$ and $\beta$ is $N'$--Morse. 
    \item If $T$ is a finite geodesic triangle in $X$ with all sides $N$--Morse, then $T$ is $\delta_N$--slim, where $\delta_N$ depends only on $N$.  
    \item If $T$ is a finite geodesic triangle in $X$ such that two sides of $T$ are $N$--Morse, then there exists an $N'$ depending only on $N$ such that the third side of $T$ is $N'$--Morse.
\end{enumerate}
\end{lemma}

In the remainder of the paper, we will be working primarily with infinite Morse geodesics, and we therefore need to generalize parts (1) and (3) of  Lemma \ref{cordes:Morse:geodesics} to the infinite case. Our strategy is to utilize Corollary \ref{inf:to:finite} to pass from infinite geodesics to finite geodesics and then appeal to Lemma \ref{cordes:Morse:geodesics}.

\begin{lemma}\label{Morse:end-points}
Let $\alpha$ be an $N$--Morse infinite geodesic and $\beta$ a $(\lambda, \varepsilon)$--quasi-geodesic such that $d_{Haus}(\alpha, \beta) < \infty$. If $\alpha$ and $\beta$ are rays, assume their initial points are the same. Then there exist a constant $D$ and a gauge $N'$  both depending only on $N$, $\lambda$, and $\varepsilon$  such that $d_{Haus}(\alpha, \beta) \leq D$ and $\beta$ is $N'$--Morse.
\end{lemma}

\begin{proof}
By Lemma \ref{lemma:close_implies_Morse}, once we bound the Hausdorff distance between $\alpha$ and $\beta$ by a constant depending only on $N, \lambda$, and $\varepsilon$, we are done. In  the ray case, Lemma \ref{var:cordes} provides a constant $D$ depending on $N, \lambda$, and $\varepsilon$ such that $\beta \subseteq \mc{N}_D(\alpha)$. If it is not the case that $\alpha \subseteq \mc{N}_D(\beta)$, then $\alpha$ is only disjoint from $\mc{N}_D(\beta)$ on intervals of finite length; thus we can apply Lemma \ref{cordes:Morse:geodesics} part (1) to show that there exists  $D'$ depending only on $N, \lambda$, and $\varepsilon$ such that $\alpha \subseteq \mc{N}_{D'}(\beta)$. In the bi-infinite case, we get the same conclusion by reducing to the ray case as follows: select a point $p$ on $\alpha$ and concatenate a geodesic from $p$ to a closest point $p'$ on $\beta$ with a subray of $\beta$ with initial point $p'$. Observe that by Lemma \ref{(3,0)-lemma}, this is a $(2\lambda+1,\varepsilon)$--quasi-geodesic with basepoint on $\alpha$. 
\end{proof}

\begin{corollary}\label{cor:Morse_end-points}
Suppose $X$ is proper. For each Morse gauge $N$, there exists $N'$ such that for all $a,b \in \partial X$ if there is an $N$--Morse geodesic from $a$ to $b$, then every geodesic from $a$ to $b$ is $N'$--Morse.
\end{corollary}

\begin{lemma}\label{2-side:Morse}
Let $T$ be a geodesic triangle (possibly with some infinite sides). For each $N,$ there exists $N'$ such that if two sides of $T$ are $N$--Morse, then the third side of $T$ is $N'$--Morse.
\end{lemma}

\begin{proof}
Let $\alpha_1, \alpha_2$, and $\alpha_3$ denote the sides of $T$, where $\alpha_1$ and $\alpha_2$ are $N$--Morse. Let $\beta$ be a $(\lambda, \varepsilon)$--quasi-geodesic with endpoints on $\alpha_3$. By Corollary \ref{inf:to:finite}, we can find $x_1$ and $x_2$ on $\alpha_3$ such that each $x_i$ is within $D$ of $\alpha_i$, where $D$ is a constant determined by $N$.  Moreover, we can choose $x_1$ and $x_2$ so that the endpoints of $\beta$ lie between $x_1$ and $x_2$ on $\alpha_3$. Similarly, there exists  $x_0$ on $\alpha_1$ such that $x_0$ is within $D$ of $\alpha_2$.

Let $\gamma_1$ be a geodesic from $x_0$ to $x_1$, $\gamma_2$ a geodesic from $x_0$ to $x_2$, and $\gamma_3$ the subsegment of $\alpha_3$ connecting $x_1$ and $x_2$ (see Figure \ref{fig:truncated_inf_triangle}). Because the endpoints of $\gamma_1$ (resp. $\gamma_2$) are within $D$ of $\alpha_1$ (resp. $\alpha_2$), and $N$ determines $D$, Lemma \ref{cordes:Morse:geodesics} part (1) says that $\gamma_1$ and $\gamma_2$ are Morse for some gauge depending only on $N$. Therefore, by Lemma \ref{cordes:Morse:geodesics} part (3), $\gamma_3$ is Morse for some gauge $N'$ depending only on $N$. By construction, the endpoints of $\beta$ lie on $\gamma_3$, and  thus we have 
\[ \beta \subseteq \mc{N}_{N'(\lambda, \varepsilon)}(\gamma_3)\subseteq \mc{N}_{N'(\lambda, \varepsilon)}(\alpha_3),\]
establishing that $\alpha_3$ is $N'$--Morse.

\begin{figure}
\begin{center}
\begin{tikzpicture}[scale=2.5]
\draw[thick] (0,0) circle (1);
\clip (0,0) circle (1);
\hgline{90}{210}
\hgline{210}{330}
\hgline{330}{450}

\node[fill=black, circle, scale=.5,label=left:{$x_0$}] (x0) at (-.075,.5){}; 
\node[fill=black, circle, scale=.5,label=below:{$x_1$}] (x1) at (-.45,-.33){}; 
\node[fill=black, circle, scale=.5,label=below:{$x_2$}] (x2) at (.45,-.32){}; 
\node(x0prime) at (.13,.5){}; 
\coordinate(x1prime) at (-.52,-.24){}; 
\node(x2prime) at (.55,-.17){}; 
\node(gamma3) at (0.0,-.2){$\gamma_3$}; 
\coordinate (betastart) at (-.25,-.285){}; 
\coordinate (betaend) at (.25,-.285){}; 

\draw[thick] (x0)--(x0prime); 
\draw[thick] (x1) --(x1prime); 
\draw[thick] (x2) --(x2prime); 
\draw[thick] (x0) to [out=290, in=20 ] node[left]{ \hspace{100pt}$\gamma_1$} (x1); 
\draw[thick] (x0) to [out=295, in=160 ] node[right]{\hspace{-7pt} $\gamma_2$} (x2); 
\draw[photon,thick] (betaend) to (betastart);  
\path (.25,-.55) to node[below]{$\beta$}(-.25,-.55); 
\end{tikzpicture}
\end{center} 
\captionsetup{width=.88\linewidth}
\caption{\small An ideal triangle with Morse sides can be truncated to a finite geodesic triangle. See proof of Lemma \ref{2-side:Morse}. } \label{fig:truncated_inf_triangle}
\end{figure}
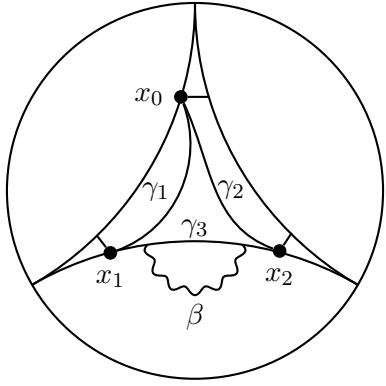 
\end{proof}

\begin{remark}\label{remark:Morse_qausi-geoedsic}
For simplicity, we have stated the results in this section in terms of Morse geodesics. However, if you replace ``$N$--Morse geodesic" with ``$N$--Morse $(\lambda',\varepsilon')$--quasi-geodesic" in any of the results, the conclusions hold with the modification that the constant $D$ or the Morse gauge $N'$ now also depends on $\lambda'$ and $\varepsilon'.$
\end{remark}

\subsection{Triangle centers} 
In this section, we develop a coarse notion of center for ideal geodesic triangles with Morse sides. Having a notion of triangle centers turns out to be the key ingredient to generalizing Paulin's notion of cross-ratio to non-hyperbolic spaces. Throughout this section, let $X$ be a proper geodesic metric space with $|\partial X| \geq 3$. 

For $n \geq 2$ and Morse gauge $N$,  let $\partial X^{(n,N)}$ denote the collection of tuples of $n$ pairwise distinct points in $\partial X$ such that every geodesic between any pair in the tuple is $N$--Morse. If $\partial X^{(3,N)} \neq \varnothing$, then for $K >0$ define 
\[m_X^{N,K}: \partial X^{(3,N)} \rightarrow 2^X\hspace{13pt} \text{by}\] 
\[(a,b,c) \mapsto  \{x \in X : x \text{ is  within } K \text{ of all sides of some geodesic triangle with vertices } a, b, c \}.\]
We call each point in the set  $m_X^{N,K}(a,b,c)$ a \textit{$K$--center} of $(a,b,c).$ 

The goal of this section is to show  for sufficiently large $K$ that $m_X^{N,K}(a,b,c)$ is non-empty  for all $(a,b,c) \in \partial X^{(3,N)}$ (Lemma \ref{lemma:internal_points_exist}) and has uniformly bounded diameter in $X$ (Lemma \ref{lemma:bounded_centers}). First, we present some notation and terminology. 

Given a triangle $T$ with vertices $a,b,$ and $c$, we will often write $T(a,b,c)$ in place of $T$ to emphasize the vertices of $T$. We let $T(a,b)$ denote the oriented side of $T$ from $a$ to $b$, and similarly denote the other sides of $T$. A \textit{$K$--internal triple} of a triangle $T$ is a set of three points, one on each side of $T$, that are pairwise at most distance $K$ apart.

\begin{lemma} \label{lemma:internal_points_exist} 
There exists a constant $K_N$ such that every geodesic triangle in $X$ with all sides  $N$--Morse has a  $K_N$--internal triple. Thus, for all $K \geq K_N$, the map $m_X^{N,K}$ sends every tuple to a non-empty set.  
\end{lemma} 

\begin{proof}
Let $T$ be a geodesic triangle with  $N$--Morse sides that are denoted by $\alpha_1,\alpha_2$ and $\alpha_3$.  By Corollary \ref{inf:to:finite}, we can find $x_1$ and $x_2$ on $\alpha_3$ such that each $x_i$ is within $D$ of $\alpha_i$, where $D$ is a constant determined by $N$. Similarly, there exists  $x_0$ on $\alpha_1$ such that $x_0$ is within $D$ of $\alpha_2$. Let $\gamma_1$ be a geodesic from $x_0$ to $x_1$, $\gamma_2$ a geodesic from $x_0$ to $x_2$, and $\gamma_3$ the subsegment of $\alpha_3$ connecting $x_1$ and $x_2$. Because the endpoints of $\gamma_1$ (resp. $\gamma_2$) are within $D$ of $\alpha_1$ (resp. $\alpha_2$), and $N$ determines $D$, Lemma \ref{cordes:Morse:geodesics} part (1)  says that $\gamma_1$ and $\gamma_2$ are Morse with gauge depending only on $N$.  Thus there exists an $N'$, depending only on $N$, such that $\gamma_1 \cup \gamma_2 \cup \gamma_3$ is a finite geodesic triangle with $N'$--Morse sides. Hence, by Lemma \ref{cordes:Morse:geodesics} part (2),  there exists $\delta_{N'}$ depending only on $N'$ such that the geodesic triangle $\gamma_1 \cup \gamma_2 \cup \gamma_3$ is 
$\delta_{N'}$--slim. This allows us to select $p_i$ on  $\gamma_i$ so that $d(p_i,p_j)\leq 4 \delta_{N'}$ for all pairs $i,j$ (see \cite{BH} III.H, Proposition 1.17). Because for all $i$, the endpoints of $\gamma_i$ are each within $D$ of the $N$--Morse geodesic $\alpha_i$, there exists $q_i$ on $\alpha_i$ such that $d(q_i,p_i) \leq N(1,4D)$.  
Combining all of this, for all pairs $i,j$ we find
\[d(q_i,q_j) \leq d(q_i,p_i) + d(p_i,p_j) + d(p_j,q_j) \leq 4\delta_{N'} + 2N(1,4D). \]
Therefore, if we let $K_N = 4\delta_{N'} + 2N(1,4D)$, then $\{q_1,q_2,q_3\}$ has the desired property.
\end{proof}

\begin{lemma} \label{lemma:bounded_centers} 

For all $N$ and $K$, there exists a constant $M$, depending only on $N$ and $K$, such that for all $(a,b,c) \in \partial X^{(3,N)}$  \[\diam ( m^{N,K}_X(a,b,c)) \leq M.\] 
\end{lemma} 

\begin{proof}
Without any loss of generality, we can assume $K\geq K_N$, where $K_N$ is the constant from Lemma \ref{lemma:internal_points_exist}. Choose $(a,b,c) \in \partial X^{(3,N)}$ and a geodesic triangle $T(a,b,c)$.  Let $\{i_{ac},i_{ab}, i_{bc}\}$ denote a $K_N$--internal triple of $T$. By Lemma \ref{Morse:end-points}, it is enough to show that the diameter of 
\[\mathcal{S}=\{ x \in X: x \text{ is within } K\text{ of each side of } T\}\] is bounded above by a constant depending only on $N$ and $K$. To show $\diam ( \mc{S})$ is bounded, we will prove that the distance from each point $x\in \mathcal{S}$ to some point in $\{i_{ac},i_{ab}, i_{bc}\}$  is bounded above by a constant depending only on $K$ and $N$. 

 Let $p$ and $q$ denote points on $T(a,b)$ and $T(a,c)$ respectively satisfying 
 \begin{equation} \label{eq:close_to_sides1} 
 d(x,p) \leq K \hspace{10pt} \text{ and } \hspace{10pt} d(x,q) \leq K. 
 \end{equation}  
By permuting the labels $a,b,c$ if needed, we may assume that $p$ comes before $i_{ab}$ on $T(a,b)$ and that $q$ comes after $i_{ac}$ on $T(a,c)$. Now the concatenation of a geodesic from $i_{ab}$ to $i_{ac}$ with the segment of $T(a,b)$ from $a$ to $i_{ab}$ is a $(1,2K_N)$--quasi-geodesic. By Lemma \ref{var:cordes}, there exists a constant $D$ determined by $N$ and a point $p'$ before $i_{ac}$ on $T(a,c)$ such that 
\begin{equation} \label{eq:close_to_sides2} d(p',p) \leq D.
\end{equation} 
Combining Inequalities (\ref{eq:close_to_sides1}) and (\ref{eq:close_to_sides2}),  we have 

\begin{equation} \label{eq:bound} d(p',q) \leq d(p',p)+d(p,x)+d(x,q) \leq D+2K.
\end{equation} 
Because $i_{ac} $ is between $p'$ and $q$ on the geodesic $T(a,c)$, Inequalities (\ref{eq:close_to_sides1}) and  (\ref{eq:bound}) imply that 
\[d(x,i_{ac}) \leq d(x,p)+d(p,p') +d(p',i_{ac}) \leq 3K+2D, \]
completing the proof. 

\end{proof} 

\subsection{Morse center spaces} \label{subsec:Morse_center_space}

Our main results will apply to Morse center spaces. We say that a proper geodesic metric space $X$ is an $N$--\textit{Morse center space} if there exists $K$ such that 
\[X= \bigcup \limits_{(a,b,c) \in \partial X^{(3,N)}} m^{N,K}_X(a,b,c).\]
That is, an $N$--Morse center space is one where every point in $X$ is a  $K$--center of some ideal geodesic triangle with $N$--Morse sides. We say that $X$ is a \textit{Morse center space} if it is an $N$--Morse center space for some gauge $N$. 

By Proposition \ref{prop:cordes_visual}, if $|\partial X| \geq 3$, then for some gauge $N$, there exists an ideal geodesic triangle $T$ with each side $N$--Morse. If $X$ also admits a cocompact group action by isometries, then every point in $X$ will be uniformly close to each side of some translate of $T$, and thus $X$ is a Morse center space. In particular, every finitely generated group equipped with a word metric is a Morse center space as long as its Morse boundary has at least three points.

\subsection{Cross-ratios and quasi-m\"obius and 2--stable maps} \label{subsec:quasimobius} 

Let $X$ and $Y$ be proper geodesic metric spaces and $N$ a Morse gauge. 
For $(a,b,c,d) \in \partial X^{(4,N)}$, define its \emph{$N$--cross-ratio} to be
\[ [a,b,c,d]_N = \diam ( m_X^{N,K_N}(a,b,c) \cup m_X^{N,K_N}(a,d,c)),\]
where $K_N$ is the constant in Lemma \ref{lemma:internal_points_exist}.
This definition of cross-ratio generalizes both the cross-ratio for CAT(0) spaces (see \cite{CM} Lemma 3.7) and hyperbolic spaces \cite{paulin} (see Section \ref{sec:paulin_cross_ratio}). 

A map $h:\partial {X} \rightarrow \partial Y$ is \emph{2--stable} if for every $N$, there exists an $N'$ such that \[h(\partial {X}^{(2,N)}) \subseteq \partial {Y}^{(2,N')}.\] 
We say that a 2--stable map $h:\partial X \rightarrow \partial Y$ is \emph{quasi-m\"obius} if  for every pair of Morse gauges $N$ and $N'$ satisfying $h(\partial X^{(2,N)}) \subseteq \partial Y^{(2,N')},$ there exists a continuous increasing function \linebreak $\psi:[0,\infty) \rightarrow [0,\infty)$ such that for all $(a,b,c,d) \in \partial X^{(4,N)}$ we have 
\[ [h(a), h(b), h(c), h(d)]_{N'}  \leq \psi([a,b,c,d]_N). \] 

\begin{remark}
Let $p,p' \in X$ be two choices of basepoint and $\phi:\partial X_{p'} \rightarrow \partial X_{p}$ the canonical homeomorphism. Because $\phi(\partial X^{(2,N)}_{p'})=\partial X^{(2,N)}_p$, a homeomorphism $h:\partial X_p \rightarrow \partial Y$ is $2$--stable if and only if $h \circ \phi$ is.  Further observe that because the sets of $K$--centers for $(a,b,c)$ and $(\phi(a), \phi(b), \phi(c))$ are identical, the $N$--cross-ratio is basepoint independent. Thus, $h$ is quasi-m\"obius if and only if $h \circ \phi$ is. 
\end{remark}

\subsection{Homeomorphisms that preserve distance between centers} \label{subsec:controlled_distortion}
In this section, we will show that in spaces with the small cross-ratio property (Definition \ref{def:small_cross} below), $2$--stable, quasi-m\"obius boundary homeomorphisms increase the distances between \linebreak $K$--centers a controlled amount (Lemma \ref{well:def}). Our approach is a generalization of the one taken in \cite{CM}, and just as in the CAT(0) setting, this is the critical step in being able to build a quasi-isometry from a quasi-m\"obius boundary homeomorphism. In this section, let $X$ and $Y$ be proper geodesic metric spaces.

\begin{definition} \label{def:small_cross} 
We say that  $X$ has the \emph{small cross-ratio property} if for each Morse gauge $N$, there exists  $C>0$ such that for any $(a,b,c,d) \in \partial X^{(4,N)}$ at least one of $[a,b,c,d]_N$,  $[a,c,b,d]_N$, or $[b,a,c,d]_N$ is less than $C$.
\end{definition}

If $[a,b,c,d]_N < C$, the $K_N$--centers of $(a,b,c)$ and $(a,d,c)$ are $C$--close, and we call the move from  $(a,b,c)$ to $(a,d,c)$ a \emph{small flip}. If $X$ has the small cross-ratio property, then given $(a,b,c,d) \in \partial {X}^{(4,N)}$, we can always use a small flip to move from the tuple $(a,b,c)$ to one of $(a,d,c)$, $(a,b,d)$, or $(b,d,c)$. Examples of spaces with the small cross-ratio property include hyperbolic spaces (Corollary \ref{cor:paulin1}, Lemma \ref{lem:paulin2}), CAT(0) spaces (\cite{CM} Lemma 4.1), and hierarchically hyperbolic groups (Proposition \ref{small:cross-ratio}).

Lemma \ref{well:def} below  is the generalization of  \cite{CM} Proposition 4.2 to all proper metric spaces. While stated in terms of CAT(0) spaces, the core of the argument presented by Charney-Murray does not utilize the CAT(0) hypothesis. After proving the following technical lemma, we will be able to follow their proof verbatim. 

\begin{lemma}\label{3-6:Morse}
Let $(a,b,c),(u,v,w) \in \partial{X}^{(3,N)}$ and $L \geq 0$ and $K \geq K_N$. If 
\[\diam( m_X^{N,K}(a,b,c) \cup m_X^{N,K}(u,v,w)) \leq L,\] then there exists $N'$ depending only on $N$, $L,$ and $K$ such that any geodesic with both endpoints in $\{a,b,c,u,v,w\}$ is $N'$--Morse.
\end{lemma}

\begin{proof}
Without loss of generality, it suffices to just check for any geodesic between $a$ and $u$.  Fix  $x \in m_X^{N,K}(a,b,c)$ and  $y \in m_X^{N,K}(u,v,w)$. Choose  an $N$--Morse geodesic $\alpha$ between $a$ and $b$ so that there is a point $x'$ on $\alpha$ within $K$ of $x$.   Let $\alpha'$ be the $N$--Morse subgeodesic of $\alpha$ from $x'$ to $a$. Now choose an $N$--Morse geodesic $\beta$ between $u$ and $v$ so that  there is a point $y'$ on $\beta$ within $K$ of $y$.  

Let $\beta'$ be the concatenation of a geodesic between $y'$ and $x'$ with any geodesic ray from $x'$ to $u$. By the triangle inequality, $d_{{X}}(x',y') \leq 2K+L$, so $\beta'$ is a $(1,2(2K+L))$--quasi-geodesic. Thus Lemma \ref{Morse:end-points} tells us that there exists a Morse gauge $N' \geq N,$ depending only on $N$, $K,$ and $L$, such that the subray of $\beta'$ from $x'$ to $u$ is $N'$--Morse.  Because $\alpha$ (and thus $\alpha'$) is also $N'$--Morse,  any geodesic from $a$ to $u$ is the third side of a triangle with two $N'$--Morse sides and hence  $N''$--Morse by Lemma \ref{2-side:Morse}, where $N''$ depends only on $N$, $K,$ and $L$.

\end{proof}
To be self-contained, we now present the Charney-Murray proof of Lemma \ref{well:def}. 
\begin{lemma}\label{well:def}
Suppose $X$ has the small cross-ratio property and  $h: \partial {X} \rightarrow \partial {Y}$ is a $2$--stable, quasi-m\"obius homeomorphism.  Suppose $h(\partial X^{(3,N)}) \subseteq \partial Y^{(3,N_h)}$ for some Morse gauges $N$ and $N_h$. Then for any $L\geq 0$ and $K \geq K_N$, there exists a constant $B$, depending on $N, N_h, K$, and $L$ only, such that for any $(a,b,c) , (u,v,w) \in \partial X^{(3,N)}$ if
\begin{equation} \label{eq:center_bound}  \diam_X (m_X^{N,K}(a,b,c)\cup m_X^{N,K}(u,v,w)) \leq L,
\end{equation} 
then
\[\diam_Y ( m_Y^{N_h,K}(h(a), h(b), h(c))\cup m_Y^{N_h,K}(h(u), h(v),  h(w))) \leq B.\]
\end{lemma}

\begin{proof}
If either $m_Y^{N_h,K}(h(a), h(b), h(c))$ or $m_Y^{N_h,K}(h(u), h(v),  h(w))$ is empty, then we are done by Lemma \ref{lemma:bounded_centers}. Suppose both are non-empty. By Lemma \ref{3-6:Morse}, there exists $N'$ depending only on $N$, $K$, and $L$ such that $N\leq N'$ and any geodesic with both endpoints in $\{a,b,c,u,v,w\}$ is $N'$--Morse.  Note that $m_X^{N',K}(a,b,c) = m_X^{N,K}(a,b,c)$ because every geodesic between a pair of $a,b,c$ is $N$--Morse. Similarly, $m_X^{N',K}(u,v,w) = m_X^{N,K}(u,v,w).$ 
So Lemma \ref{lemma:bounded_centers} together with assumption (\ref{eq:center_bound}) implies that there exists $L'$ depending only on $L$ and $N'$ (and thus only on $L, K$, and $N$) such that 
\begin{equation} \label{eq:bound_on_center_sets}
\diam_X(m_X^{N',K_{N'}}(a,b,c) \cup m_X^{N',K_{N'}}(u,v,w)) \leq L'.
\end{equation} 
 Let $C$ be the constant given by the small cross-ratio property for $N'$, and for convenience, let $[a,b,c,d] = [a,b,c,d]_{N'}$. We will obtain a sequence of tuples $V_1,\hdots,V_5\in \partial X^{(3,N')}$ such that:

\begin{itemize}
    \item Each $V_i$ has entries in $\{a,b,c,u,v,w\}$.
    \item $V_1 = (a,b,c)$ and $V_5 = (u,v,w)$.
    \item   For $1\leq i \leq 4$, $V_i$ and $V_{i+1}$ share two entries and the diameter of the union of the \linebreak $K_{N'}$--centers of $V_i$ and $V_{i+1}$ is bounded by a constant depending only on $L$, $K$, and $N.$
\end{itemize}

Let $V_1 = (a,b,c)$ and $V_5 = (u,v,w)$. Because $X$ has the small cross-ratio property, at least one of  $[a,b,c,u]$,  $[a,c,b,u]$, or $[b,a,c,u]$  is smaller than $C$.  Without loss of generality, we can assume $[a,b,c,u]$ is small, and we can perform a small flip from $(a,b,c)$ to $(a,u,c)$. Let $V_2 = (a,u,c)$. There are two cases for $V_3$ and $V_4$.

\textbf{Case 1}: If  $[a,u,c,w]$ or $[a,u,c,v]$ is at least $C$, then we can perform a small flip from $(a,u,c)$ onto one of $(a,u,w)$, $(w,u,c)$, $(a,u,v)$, or $(v,u,c)$. In this case we let $V_3$ be the tuple from the list above separated from $(a,u,c)$ by a small flip and define $V_4=(u,v,w)$. 

\textbf{Case 2}: If both $[a,u,c,w]$ and $[a,u,c,v]$ are smaller than $C$, we can perform small flips from $(a,u,c)$ onto both $(a,w,c)$ and $(a,v,c)$. 
Applying the small cross-ratio property to $(u,v,w,a)$, we define $V_4$ to be the tuple in the list 
$(u,a,w)$, $(u,v,a)$, $(a,v,w)$ that differs from $(u,v,w)$ by a small flip.  We now define $V_3$ to be whichever of $(a,w,c)$ or $(a,v,c)$ shares two entries with $V_4.$

In either case,  $\diam( m_X^{N',K_{N'}}(V_i) \cup m_X^{N',K_{N'}}(V_{i+1})) < C$ for $i=1,2,4$. This together with Inequality (\ref{eq:bound_on_center_sets}) and the triangle inequality implies that the diameter of the set of $K_{N'}$--centers of $V_i$ and $V_{i+1}$ is at most $B'=L'+3C$ for all $1\leq i\leq 4$. 

We now apply $h$ to our sequence of tuples. Because $h$ is 2--stable, there exists a gauge $N''$ such that $h(\partial X^{(2,N')}) \subseteq \partial Y^{(2,N'')}$. Because $h$ is quasi-m\"obius, for $1 \leq i \leq 4$, the diameter of the set of $K_{N''}$--centers of $h(V_i)$ and $h(V_{i+1})$ is bounded by a constant depending only on $B'$. Thus, by the triangle inequality, the diameter of the set of $K_{N''}$--centers of $h(V_1)$ and $h(V_5)$ is bounded by a constant determined by $B'$.  If $K \leq K_{N''}$, this completes the proof of the lemma. Otherwise, $K > K_{N''}$, and we can apply Lemma \ref{lemma:bounded_centers} to bound the diameter of the set of $K$--centers of $h(V_1)$ and $h(V_5)$ by a constant depending only on $K,B',$ and $N_h$, and thus only on $K, L$, $N,$ and $N_h$ as desired. 
\end{proof}

\section{Main Theorems}

\subsection{Quasi-isometries induce quasi-m\"obius, 2--stable homeomorpshisms}
\begin{theorem}\label{thrm:quasi-isometry_to_quasi-mobius} 
Let $f:X \rightarrow Y$ be a quasi-isometry between proper geodesic metric spaces. Then the induced homeomorphism $f_*:\partial X \rightarrow \partial Y$ and its inverse are $2$--stable and quasi-m\"obius. 
\end{theorem}

\begin{proof}
Let $f:X \rightarrow Y$ be a $(\lambda, \varepsilon)$--quasi-isometry and $f_*:\partial X \rightarrow \partial Y$ the homeomorphism induced by $f$ (see Proposition \ref{prop:cordes_inducedmap}). Because $f_*^{-1}$ will be induced by the correct choice of quasi-inverse for $f$, it is sufficient to only check the theorem for $f_*$. To see $f_*$ is $2$--stable, consider $(a,b) \in \partial {X}^{(2,N)}$ and let $\alpha$ be an $N$--Morse geodesic between $a$ and $b$. Pick $y$ on $f(\alpha)$. By Proposition \ref{prop:cordes_inducedmap}, there exist two $N'$--Morse geodesic rays based at $y$, one with endpoint $f_*(a)$ and the other with endpoint $f_*(b)$, where $N'$ depends only on $N, \lambda$, and $\varepsilon$. Thus, by Lemma \ref{2-side:Morse}, there exists an $N''$ depending only on $N, \lambda,$ and $\varepsilon$ such that any geodesic with endpoints $f_*(a)$ and $f_*(b)$ is $N''$--Morse, establishing that $f_*$ is 2--stable. 

To see $f_*$ is quasi-m\"obius, let $N$ and $N'$ be Morse gauges such that $f_*(\partial X^{(2,N)}) \subseteq \partial Y ^{(2,N')}$ and consider $(a,b,c,d) \in \partial X^{(4,N)}$. 

Choose ideal geodesic triangles $T(a,b,c)$ and $T(a,c,d)$ in $X$ and $T(f_*(a), f_*(b), f_*(c))$ and $T(f_*(a), f_*(c), f_*(d))$ in $Y$. By Lemma \ref{Morse:end-points}, there exists a constant $D$, depending on $\lambda, \varepsilon, $ and $N'$ only, such that each side of the quasi-geodesic triangle $f(T(a,b,c))$ is contained in the \linebreak $D$--neighborhood of the corresponding side of $T(f_*(a), f_*(b), f_*(c))$. Similarly, for $f(T(a,c,d))$. 
Thus, if $x$ and $y$  are within $K_N$ of each side of $T(a,b,c)$ and $T(a,c,d)$, respectively, then $f(x)$ and $f(y)$ are within $\lambda K_N+\varepsilon +D$ of each side of $T(f_*(a),f_*(b), f_*(c))$ and $T(f_*(a), f_*(c), f_*(d))$ respectively. 

Now define $K=\max\{\lambda K_N +\varepsilon+D, K_{N'}\}$ and recall  that Lemma \ref{lemma:bounded_centers} bounds the diameter of the image of a tuple under $m_Y^{N',K}$ by a constant $M$ depending only on $N'$ and $K$, and thus only on $N, N', \lambda,$ and $\varepsilon$. Thus, 
\begin{align} \label{eq:distpq_bound} 
[f_*(a), f_*(b),f_*(c), f_*(d)]_{N'} &\leq \diam( m_Y^{N',K}(f_*(a), f_*(b),f_*(c) ) \cup m_Y^{N',K}(f_*(a), f_*(c),f_*(d) )) \\ &\leq  d(f(x),f(y))+2M\nonumber  \leq   \lambda[a,b,c,d]_N +\varepsilon+ 2M. \nonumber 
\end{align}
\end{proof} 

\begin{remark}
The above proof also shows that being a Morse center space is a quasi-isometry invariant among proper geodesic metric spaces.
\end{remark}

\subsection{Quasi-m\"obius, 2--stable homeomorphisms induce quasi-isometries}

Throughout this section, $X$ and $Y$ will denote proper geodesic Morse center spaces with basepoints $p \in X$ and $r \in Y$. We assume there exists a $2$--stable homeomorphism $h: \partial_MX_p\rightarrow \partial_M Y_r$. We will show that, under some additional assumptions, $h$ is induced by a quasi-isometry, which we now construct.

Under our assumptions, there must exist a Morse gauge $N$  and a constant $K_0$ such that for all $K\geq K_0$
\[X = \bigcup \limits_{V \in \partial X^{(3,N)}} m_X^{N,K}(V) \hspace{5pt} \text{ and } \hspace{5pt} Y = \bigcup \limits_{V \in \partial Y^{(3,N)}} m_Y^{N,K}(V).\]
Furthermore, because $h$ is 2--stable,  $h(\partial X^{(3,N)}) \subseteq \partial Y^{(3,N_h)}$ for some Morse gauge $N_h \geq N$. 

For $K\geq \max\{K_0,K_{N_h}\}$, we define a map $f_K\colon X \rightarrow Y$ by sending $p$ to $r$,  and for each point $x \in X -\{p\}$, we choose an ideal triangle $T_x$ so that $x$ is a $K$--center of $T_x$ and the vertex tuple $(a,b,c)$ of $T_x$ is in $\partial X^{(3,N)}$.  We then define $f_K(x)$ to be a choice of $K$--center of $(h(a),h(b),h(c))$. We will now show that, under additional assumptions which we describe, $f_K$ is a quasi-isometry and that the induced map $f_K{_*}:\partial_M X_p \rightarrow \partial_M Y_r$ agrees with $h$.

\begin{theorem}\label{thrm:main_body_part1} 
Let $X$ and $Y$ be proper Morse center spaces with the small cross-ratio property. If $h:\partial X \rightarrow \partial Y $ is a $2$--stable, quasi-m\"obius homeomorphism  with 2--stable and quasi-m\"obius inverse, then $f_K:X \rightarrow Y$ is a  quasi-isometry for $K$ sufficiently large. 
\end{theorem}

\begin{proof} The following argument is very similar to the proof in the CAT(0) setting presented in Theorem 4.4 of \cite{CM}.

Because both $h$ and $h^{-1}$ are 2--stable, there exist Morse gauges $N'$ and $N''$ (both greater than $N$) such that
\[ h^{-1}(\partial Y^{(3,N_h)} ) \subseteq \partial X^{(3,N')}  \hspace{5pt}  \text{ and }  \hspace{3pt}  h(\partial X^{(3,N')} ) \subseteq \partial Y^{(3,N'')}. \]
Choose $K \geq \max\{K_0,K_N,K_{N_h},K_{N'},K_{N''}\}$ and  define maps $g \colon X \rightarrow 2^Y$ and $\hat{g}: Y \rightarrow 2^X$ by:
\[g(x)=m^{N_h,K}_Y \circ h \circ (m_X^{N,K})^{-1}(x) \hspace{10pt} \text{and} \hspace{10pt} \hat{g}(y) = m^{N',K}_X \circ h^{-1} \circ (m_Y^{N_h,K})^{-1}(y). \]

\noindent Lemmas \ref{lemma:internal_points_exist}, \ref{lemma:bounded_centers},  and \ref{well:def} imply that $g$ and $\hat{g}$ are well defined coarse maps. We further assume $K$ is large enough so that $q \in g(p)$. Observe that $f_K$ is a choice of induced function for $g$. Thus showing that $f_K$ is a quasi-isometry is equivalent to showing the following:
\begin{enumerate}[(1)]
    \item There exist constants $A\geq1$ and $B\geq 0$ such that $d_Y(g(x),g(x')) \leq Ad_X(x,x')+B$ and $d_X(\hat{g}(y),\hat{g}(y')) \leq Ad_Y(y,y')+B$ for all $x,x'\in X$ and $y,y' \in Y$.
    \item  $\hat{g}\circ g$ and $g\circ \hat{g}$ are bounded distance away from the identity.
\end{enumerate}

\noindent \textbf{Proof of (1):} Let $x,x' \in X$. Select a sequence of points $x = x_0,x_1,\ldots, x_n=x'$  along a geodesic connecting $x$ and $x'$ so that $d_X(x_i,x_{i+1}) = 1$ for $0\leq i \leq n-2$ and $d_X(x_{n-1}, x_n) \leq 1$. Let $M$ be the constant from Lemma \ref{lemma:bounded_centers} bounding the images of $m_X^{N,K}$, and  let $B$ be the constant from Lemma \ref{well:def} corresponding to $L=2M+1$. Then for $0 \leq i \leq n-1$ we have $d_Y(g(x_i),g(x_{i+1})) \leq B$, and thus
\[d_Y(g(x),g(x') )\leq Bn \leq B(d_X(x,x') +1).\]

\noindent Because Lemma \ref{well:def} also applies to $h^{-1}$, the proof of  (1) is completed by an analogous argument for $\hat{g}$.\\

\noindent \textbf{Proof of (2):} Let $x \in X$ and $(a,b,c) \in (m_X^{N,K})^{-1} (x)$. Then $ m_Y^{N_h,K}(h(a), h(b), h(c)) \subseteq g(x)$, implying that  
\begin{equation} 
\label{eq:bounded_from_id} 
x \in m_X^{N',K} (a,b,c) \subseteq \hat{g}(g(x)).
\end{equation}
Thus because $\hat{g}$ and $g$ are coarse maps, (\ref{eq:bounded_from_id}) implies that $\hat{g}\circ g$  is bounded distance from the identity. 
For the $ g \circ \hat{g}$ case, define \[\tilde{g} = m^{N'',K}_Y \circ h \circ (m^{N',K}_X)^{-1}.\] By an  argument analogous to the $\hat{g} \circ g$ case above, we have that $y \in \tilde{g} (\hat{g}(y))$ for all $y \in Y$. Thus we only need to uniformly bound the distance between $g(x)$ and $\tilde{g}(x)$ for all $x \in X$.  To do so, observe that  $\partial X^{(3,N)} \subseteq \partial X^{(3,N')} $ and $\partial Y^{(3,N_h)} \subseteq \partial Y^{(3,N'')}$. Thus,  $g(x) \subseteq \tilde{g}(x)$. This completes the proof of (2) and thus establishes that $f_K$ is a quasi-isometry. 

\end{proof}

\begin{theorem} \label{thrm:main_body_part2} 
Let $X$ and $Y$ be proper Morse center spaces with the small cross-ratio property. 
Suppose $h:\partial X \rightarrow \partial Y $  is $2$--stable and quasi-m\"obius with 2--stable and quasi-m\"obius inverse. Then for all $K$ sufficiently large, $h$ is induced by  $f_K:X \rightarrow Y$. 
\end{theorem} 

\begin{proof} 
Let $K$ be large enough so that Theorem \ref{thrm:main_body_part1} implies $f_K:X\rightarrow Y$ is a quasi-isometry, and let $f=f_K$. Choose a point $q \in \partial X$ and a geodesic ray $\gamma$ based at $p$ representing $q$. To prove $h$ and the boundary map $f_*$ induced by $f$ are the same map, we must show $f_*(q)=h(q)$. 

Choose a sequence of points $(x_n)$ on $\gamma$ satisfying $\lim \limits_{n \rightarrow \infty} d(p,x_n) = \infty$. Let $a_n, b_n,$ and  $c_n$ denote the vertices of the triangle $T_{x_n}$ with $K$--center $x_n$ used to define $f(x_n)$. To proceed, we require the following claims.

\vspace{5pt} 
\noindent \textbf{Claim 1:} 
\textit{There exists a Morse gauge $N'$ such that $a_n,b_n,c_n \in \partial^{N'} X$ for all $n$. Furthermore, at least two of $(a_n), (b_n), $ and $(c_n)$ converge to $q$ in $\partial X$. }

\vspace{5pt} 

\noindent \textbf{Claim 2:}  
\textit{At least two of $(h(a_n)),( h(b_n)),$ and $(h(c_n))$ converge to $f_*(q)$ in $\partial Y$.}

\vspace{5pt}

For now, assume Claims 1 and 2 hold.  Then without loss of generality, we may assume  $(a_n)$ converges to $q$ in $\partial X$  and $(h(a_n))$ converges to $f_*(q)$ in $\partial Y$.  The continuity of $h$ implies that $(h(a_n))$ converges to $h(q)$ in $\partial Y$. We would like to conclude that $h(q)=f_*(q)$, but it is unknown if $\partial  Y$ is Hausdorff. To finish we therefore need the following argument. 

There exists a Morse gauge $N''$ such that
\begin{equation} \label{eq:strata_to_strata}
h(\partial^{N'} X ) \subseteq \partial^{N''} Y. 
\end{equation} 
Together with Claim 1, this implies that $h(a_n)$ is in $\partial^{N''}Y$ for all $n$. We can now apply Lemma \ref{lemma:convergence_in_strata} to conclude that $h(a_n)$ converges to both $h(q)$ and $f_*(q)$  in $(\partial^{N''} Y, \T_v^{N''})$.  Because $(\partial^{N''} Y, \T_v^{N''})$ is Hausdorff (Proposition \ref{prop:fundneigh}), $f_*(q)=h(q)$, as desired. It remains to prove Claims 1 and 2. 

\noindent \textbf{Proof of Claim 1:}
Choose a geodesic triangle $T(p, a_n, x_n)$ so that the edge from $p$ to $x_n$ is a subpath of $\gamma$. 
 Because $x_n$ is a $K$--center of $T_{x_n}$, concatenating $T(x_n, a_n)$ with a geodesic from $x_n$ to a nearest point  on $T_{x_n}(a_n,c_n)$ yields a $(1,2K)$--quasi-geodesic. Thus by  Lemma \ref{Morse:end-points}, for some $D$ depending only on $N$ and $K$, we have 
\begin{equation} \label{eq:neigh_containment} 
T(x_n, a_n) \subseteq \mathcal{N}_D(T_{x_n}(a_n, c_n)),
\end{equation}
and 
 $T(x_n, a_n)$ is Morse for some gauge determined by $N$ and $K$. 
It then follows from Lemma \ref{2-side:Morse} that each side of $T(p, a_n,x_n)$ is $N'$--Morse for some gauge $N'$ depending only on $N$, $K$, and the gauge of $\gamma$. So, $q, a_n \in \partial^{N'} X$, and by similar arguments $b_n, c_n \in \partial^{N'} X$. Thus the first part of Claim 1 is proved. 

Suppose the second half of Claim 1 is false. Without loss of generality, we may assume that $(a_n)$ and $(b_n)$ do not converge to $q$ in $\partial X$. 

Let $\{i_{px_n}, i_{pa_n}, i_{x_na_n}\}$ be a  $K_{N'}$--internal triple for $T(p,x_n,a_n)$, where $K_{N'}$ is the constant determined by $N'$ from Lemma \ref{lemma:internal_points_exist}. To simplify notation, we let $\alpha_n$ denote $T(p, a_n)$. 
The concatenation of the segment of $\alpha_n$ from $p$ to $i_{pa_n}$ with a geodesic segment from $i_{pa_n}$ to $i_{px_n}$ is a $(1, 2K_{N'})$--quasi-geodesic. 
So because $\gamma$ is $N'$--Morse, we have that 
\[d(\alpha_n(t), \gamma) \leq N'(1,2K_{N'}) \text{ for all } 0\leq t \leq d(p, i_{pa_n}),\]
which implies that 
\begin{equation} \label{eq:fellow_travel} d(\alpha_n(t), \gamma(t)) \leq 2N'(1,2K_{N'}) \text{ for all } 0 \leq t \leq d(p, i_{pa_n}).
\end{equation} 

 Because $(a_n)$ does not converge to $q$ in $\partial X$, it does not converge to $q$ in $(\partial^{\ol{N}} X, \T_v^{\ol{N}})$ for any Morse gauge $\ol{N}$. Thus, Inequality (\ref{eq:fellow_travel}) and Proposition \ref{prop:fundneigh} imply that there exists a constant $L$ such that after passing to a subsequence we have that $d(p,i_{pa_n}) \leq L$ for all $n$. Combining this with (\ref{eq:neigh_containment}), we find that for all $n$ 
\[d(p,T_{x_n}(a_n,c_n)) \leq d(p, i_{pa_n})+d(i_{pa_n}, i_{x_na_n})+d(i_{x_na_n}, T_{x_n}(a_n,c_n))  \leq L+K_{N'}+D.\]
We can repeat the above argument replacing $c_n$ with $b_n$ to conclude that, after passing to a subsequence, $d(p,T_{x_n}(a_n,b_n)) \leq L+K_{N'}+D$ for all $n$. Furthermore, because $(b_n)$  does not converge to $q$ in $\partial X$, a similar argument shows that after passing to a subsequence, $d(p,T_{x_n}(b_n,c_n))$ is also bounded above by $L+K_{N'}+D$ for all $n$. 

So after passing to a subsequence, both $p$ and $x_n$ are $(L+K_{N'}+D+K) $--centers  of $(a_n,b_n,c_n)$  for every $n$. But this cannot be the case because $\lim \limits_{n \rightarrow \infty} d(p,x_n)=\infty$, contradicting Lemma \ref{lemma:bounded_centers}, which says center sets have bounded diameter. Therefore, at least two of $(a_n), (b_n), $ and $(c_n)$ converge to $q$ in $\partial X$, completing the proof of Claim 1. 

\noindent \textbf{Proof of Claim 2:} 
The points $h(a_n), h(b_n), $ and $h(c_n)$ are in $\partial ^{N''}Y$ for all $n$ by  (\ref{eq:strata_to_strata}). 
By definition of $f_*$, the quasi-geodesic $f(\gamma)$ stays bounded distance from a Morse geodesic ray $\ol{\gamma}$ based at $f(p)$ representing $f_*(q)$. Thus, for each $n$, there exists a point $\ol{f(x_n)}$ on $\ol{\gamma}$ that is a uniformly bounded distance from $f(x_n)$. Because $f$ is a quasi-isometry and because $\lim \limits_{n\rightarrow \infty} d(p, x_n)=\infty$, we have $\lim \limits_{n \rightarrow \infty}d(f(p), \ol{f(x_n)})=\infty$. Moreover, by the definition of $f$, each $f(x_n)$ is  a $K$--center of $(h(a_n), h(b_n),h(c_n))$, which implies that there exists a constant $K'$ (independent of $n$) such  that  $\ol{f(x_n)}$ is a $K'$--center of $(h(a_n), h(b_n), h(c_n))$. 

 We can now apply an argument like that used to prove Claim 1 to see that at least two of $(h(a_n)), (h(b_n)),$ and $(h(c_n))$ converge to $f_*(q)$ in $\partial Y$. 
\end{proof}

\section{Applying main results to HHSs} \label{section:HHSs}
The goal of this section is to show that Theorems \ref{thrm:main_body_part1} and \ref{thrm:main_body_part2} apply to maximal almost HHSs.  In Section \ref{sec:paulin_cross_ratio}, we generalize results about Paulin's cross-ratio to non-proper hyperbolic spaces. Then in Section \ref{sec:HHS_smallcross}, we use these results to show that maximal almost HHSs have the small cross-ratio property, thus achieving the goal. 

\subsection{Paulin's cross-ratio for hyperbolic spaces} \label{sec:paulin_cross_ratio}

Let $X$ be a $\delta$--hyperbolic metric space. Paulin defined the cross-ratio of a tuple of four distinct points $(a,b,c,d)$ in $\partial X$ to be 
\[[a,b,c,d]=\sup   \frac{1}{2} \liminf \limits_{n \rightarrow \infty} (d(a_n,d_n)+d(b_n,c_n)-d(a_n,b_n)-d(c_n,d_n) ), \]
where the supremum is taken over all sequences $(a_n)$, $(b_n)$, $(c_n)$, $(d_n)$ in $X$ representing $a, b,c,$ and $d$, respectively. 
At first glance, our definition of cross-ratio may seem drastically different from Paulin's original definition.
However, the following proposition of Paulin demonstrates that the definitions are in coarse agreement when $X$ is a proper hyperbolic space. 

In what follows, to simplify notation,  $m_X^K(a,b,c)$ will denote $m_{X}^{N,K}(a,b,c)$, where $N$ is the uniform Morse gauge for geodesics in $X$ that depends only on $\delta$.

\begin{proposition} [\cite{paulin} Lemma 4.2]\label{prop:paulin} Let $X$ be a proper $\delta$--hyperbolic space.  There exists $K_0>0$ depending only on $\delta$ such that for any $K\geq K_0$, there exists a constant $Q>0$ such that for any distinct points $a,b,c,d\in \partial X$
\[|[a,b,c,d]| \stackrel{1,Q}{\asymp} \diam(m_{X}^{K}(a,b,c) \cup m_{X}^{K}(a,d,c)).\]
 \end{proposition}

Using Proposition \ref{prop:paulin}, Paulin \cite{paulin}  proves that proper hyperbolic spaces have the small cross-ratio property.

\begin{corollary}[\hspace{-4.5pt} \cite{paulin} Lemma 4.2]\label{cor:paulin1}
Let $X$ be a proper $\delta$--hyperbolic space. There exists $C>0$ such that for any distinct points $a,b,c,d\in \partial X$, at least one of $|[a,b,c,d]|$, $|[a,c,b,d]|$, or $|[b,a,c,d]|$ is less than $C$. 

\end{corollary}

In Section \ref{sec:HHS_smallcross}, we will prove that maximal almost hierarchically hyperbolic spaces have the small cross-ratio property.  To do so we would like to apply Corollary \ref{cor:paulin1} to the hyperbolic space $CS$ corresponding to the maximal element in the HHS strcture; however, $CS$ may not be proper. For the sake of Corollary \ref{cor:paulin1}, the major difference between proper and non-proper hyperbolic spaces is that in the latter case, the boundary is not visual (i.e., if $X$ is a non-proper $\delta$--hyperbolic space, then there does not have to exist a bi-infinite geodesic between every pair of boundary points). Fortunately, in the non-proper case, there does exist a $(1,20\delta)$--quasi-geodesic between every pair of boundary points \cite{KB}.  Thus we can modify the definition of  $m_{X}^{K}$ slightly and for distinct points $a,b,c \in \partial X$ define 
\[
\hat{m}_X^K(a,b,c)= \left\{ \begin{tabular}{m{1cm} m{.01cm} m{7cm} }
$x \in X$ &:&  $x$  is  within $K$ of all sides of some $(1,20\delta)$--quasi-geodesic  triangle with vertices $a,b,c$
\end{tabular} \right\}.
\]
Since all $(1,20\delta)$--quasi-geodesics in a hyperbolic space are uniformly Morse, nearly identical arguments to Lemmas \ref{lemma:internal_points_exist} and \ref{lemma:bounded_centers} yield the following result. 

\begin{lemma} \label{lem:bounded_modified_centers}  There exists $K_0$ depending only on $\delta$ such that for all $K \geq K_0$, the set $\hat{m}^K_X(a,b,c)$ is non-empty with diameter bounded in terms of $\delta$ and $K$ for all distinct points $a,b,c \in \partial X$. 
\end{lemma} 

We now recover Proposition \ref{prop:paulin} and Corollary \ref{cor:paulin1} for non-proper hyperbolic spaces using an argument similar to the one given by Paulin.

\begin{lemma}\label{lem:paulin2}
Let $X$ be a (not necessarily proper) $\delta$--hyperbolic space and  $a,b,c,d \in \partial X$ be distinct points. Then
\begin{enumerate}[(1)]
    \item There exists $K_0>0$ depending only on $\delta$ such that for any $K\geq K_0$, there exists $Q>0$ such that $|[a,b,c,d]| \stackrel{1,Q}{\asymp} \diam(\hat{m}_{X}^{K}(a,b,c) \cup \hat{m}_{X}^{K}(a,d,c)).$

    \item There exists $C>0$ depending only on $\delta$, such that at least one of $|[a,b,c,d]|$, $|[a,c,b,d]|$, or $|[b,a,c,d]|$ is less than $C$. 
\end{enumerate}
\end{lemma}

\begin{proof}
In this proof only, we extend the map $m_X^K$ to include interior points by defining
\[
m_X^K(x,y,z)= \left\{ \begin{tabular}{m{1cm} m{.01cm} m{7cm} }
$v \in X$ &:&  $v$  is  within $K$ of all sides of some geodesic  triangle with vertices $x,y,z$
\end{tabular} \right\}
\]
for all distinct points $x,y,z \in X$.   
After replacing bi-infinite geodesics with $(1,20\delta)$--quasi-geodesics, Paulin's proofs of Proposition \ref{prop:paulin} and Corollary \ref{cor:paulin1} will work as written to prove (1) and (2) if we can prove the following claim:\\

\noindent \textbf{Claim:} Let $(a_n),(b_n),(c_n),$ and $(d_n)$ be sequences in $X$ representing $a,b,c,d \in \partial X$ respectively.  There exists $K_0>0$ depending only on $\delta$ such that for each $K\geq K_0$
\[\diam(m^K_X(a_n, b_n, c_n) \cup \hat{m}^K_X(a,b,c))\] is bounded by a constant depending only on $K$ and $\delta,$  for all but a finite number of $n$.\\

\noindent {\bf Proof of Claim:} Let  $T(a,b,c)$ be a $(1,20\delta)$--quasi-geodesic triangle with vertices $a,b,$ and $c$. There exists $\ol{K}$ depending only on $\delta$ such that $T$ has a $\ol{K}$--internal triple, which we denote by $\{i_{ab}, i_{ac}, i_{bc}\}$.

Let $a_n'$ and $b_n'$ denote the closest point projection of $a_n$ and $b_n$ respectively onto $T(a,b)$. Because $(a_n)$ and $(b_n)$ converge to distinct points in $\partial X$,   
by taking $n$ sufficiently large, we may assume that $i_{ab}$ is between $a_n'$ and $b_n'$ on $T(a,b)$ (see Figure \ref{fig:paulin_non_proper}). Let $\psi$ be the concatenation of a geodesic from $a_n$ to $a_n'$, the subsegment of $T(a,b)$ from $a_n'$ to $b_n'$, and a geodesic from $b_n'$ and $b_n$. By taking $n$ to be sufficiently large, we can guarantee that $b_n'$ is closer to $b_n$ than any point on the geodesic from $a_n$ to $a_n'$. We can then apply Lemma \ref{(3,0)-lemma} twice to conclude that $\psi$ is a $(7,20\delta )$--quasi-geodesic. Thus, for any choice of geodesic triangle $T(a_n,b_n,c_n)$ we have 
\begin{equation} \label{eq:internal_point_close1} d(i_{ab},T(a_n,b_n)) \leq N(7,20\delta).
\end{equation} 
Similarly, for all $n$ sufficiently large, we have
\begin{equation} \label{eq:internal_point_close2}  d(i_{ac}, T(a_n,c_n)) \leq N(7,20\delta)  \hspace{7pt} \text {and} \hspace{7pt}  d(i_{bc}, T(b_n,c_n)) \leq N(7,20\delta).
\end{equation} 

Choose $K \geq \ol{K}+2N(7,20\delta).$ Because $\diam \{i_{ab}, i_{bc}, i_{ac}\} \leq \ol{K}$, Inequalities (\ref{eq:internal_point_close1}) and (\ref{eq:internal_point_close2}) and the triangle inequality imply that $i_{ab} \in m_X^{K} (a_n,b_n,c_n)$. 
It is a classical fact about hyperbolic spaces that the diameter of $m_K^X(a_n,b_n,c_n)$ is bounded by a constant depending only on $K$ and $\delta$, and by Lemma \ref{lem:bounded_modified_centers},  the same is true of $\hat{m}^K_X(a,b,c)$. This together with the fact that $i_{ab} \in \hat{m}^K_X(a,b,c)$  implies that $\diam(m^K_X(a_n, b_n, c_n) \cup \hat{m}^K_X(a,b,c)) $ is bounded by a constant depending only on $K$ and $\delta$, completing the proof of the claim. 

\begin{figure}
\center 
\def\svgwidth{160pt}
 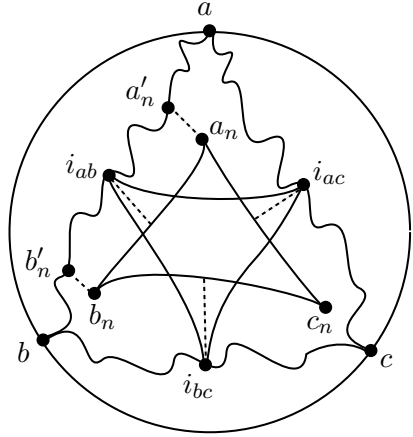
 \captionsetup{width=.85\linewidth}
 \caption{\small Centers of finite triangles approximate centers of ideal triangles. See proof of Lemma \ref{lem:paulin2}. } \label{fig:paulin_non_proper}
\end{figure} 

\end{proof}
\subsection{Maximal almost HHSs have the small cross-ratio property} \label{sec:HHS_smallcross}
Throughout this section, $(\mc{X}, \mf{S})$ will denote a maximal almost HHS, and $CS$ will denote the hyperbolic space corresponding to the unique maximal element $S\in \mf{S}$. 
In this section, we prove that $\mc{X}$ has the small cross-ratio property, and thus Theorems \ref{thrm:main_body_part1} and \ref{thrm:main_body_part2} apply to $\mc{X}$ as long as it is a Morse centered space. 
\begin{proposition} \label{small:cross-ratio}
If $\mc{X}$ is a proper geodesic metric space that can be equipped with a maximal almost HHS structure, then for each Morse gauge $N$, there exists a constant $C$ such that for any $(a,b,c,d) \in \partial \mc{X}^{(4,N)}$ at least one of $[a,b,c,d]_N$, $[a,c,b,d]_N $, or $[b,a,c,d]_N$ is less than $C$.
\end{proposition}

Paulin showed that proper hyperbolic spaces have the small cross-ratio property (Corollary \ref{cor:paulin1}), and we have generalized this result to non-proper hyperbolic spaces (Lemma \ref{lem:paulin2}).
While $\mathcal{X}$ is not necessarily hyperbolic, Theorem \ref{thrm:HHS_Stability} gives us an injection $\ol{\pi}_S$ from the Morse boundary of $\mc{X}$ to the Gromov boundary  of $ CS$.  Thus our strategy for proving Proposition \ref{small:cross-ratio} is to relate the notion of cross-ratio in $\partial \mc{X}$ we defined in Section \ref{subsec:quasimobius}  to Paulin's definition of cross-ratio in $CS$ and then apply Lemma \ref{lem:paulin2}. 

 For $(a,b,c,d) \in \partial \mc{X}^{(4,N)}$, define
\[[a,b,c,d]_S= [\ol{\pi}_S(a), \ol{\pi}_S(b), \ol{\pi}_S(c), \ol{\pi}_S(d) ], \]
where the right hand side denotes Paulin's cross-ratio. We will show that $|[a,b,c,d]_S|$ is coarsely equal to our cross-ratio $[a,b,c,d]_N$.

\begin{proposition}\label{relate:cross-ratios}
Let $(\mc{X}, \mf{S})$ be a maximal almost HHS. For every Morse gauge $N$, there exists $Q$ depending only on $N$ and the almost HHS structure such that
\[ [a,b,c,d]_N  \stackrel{Q,Q}{\asymp } |[a,b,c,d]_S| \hspace{10pt} \text{for all } (a,b,c,d) \in \partial \mc{X}^{(4,N)}.\]
\end{proposition}

\begin{proof} 
For simplicity of notation, in this proof define 
 \[d_{U}(x,x') = d_{CU} (\pi_U(x), \pi_U(x')) \hspace{7pt} \text{ for } x, x' \in \mc{X} \text{ and } U \in \mf{S}.\]
Let $\delta$ denote the hyperbolicity constant of $CS$.  Pick a Morse gauge $N$ and $(a,b,c,d) \in \partial \mc{X}^{(4,N)}$. Select $x_1 \in m_\mc{X}^{N, K_N}(a,b,c)$ and $x_2 \in m_\mc{X}^{N,K_N}(a,c,d)$. We will argue that
\begin{equation} \label{eq:string_coarse_eq} 
[a,b,c,d]_N \asymp d_{\mc{X}}(x_1,x_2) \asymp d_S(x_1,x_2) \asymp | [a,b,c,d]_S|,
\end{equation}
where the constants associated to each $\asymp$ depend only on $N$ and the almost HHS structure on $\mc{X}$. 

 By Lemma \ref{lemma:bounded_centers}, the image of a tuple under $m_\mc{X}^{N,K_N}$ has diameter uniformly bounded by a constant depending only on $N$. Thus, the first coarse equality in Equation (\ref{eq:string_coarse_eq}) holds. 
 
 To prove the second coarse equality in Equation (\ref{eq:string_coarse_eq}), we will show that there exists  $B$ depending only on $N$ such that $d_U(x_1, x_2) \leq B$ for all $U \in \mf{S} - \{S\}$ and then invoke Theorem \ref{dist:formula} with $\sigma > B$. 
Let $\gamma$ be an $N$--Morse geodesic in $\mc{X}$ from $a$ to $c$. By Theorem \ref{thrm:HHS_Stability},  there exists  $B'$ depending only on $N$ such that $\diam(\pi_U(\gamma)) \leq B'$ for all $U \in \mf{S} - \{S\}$. This together with Lemma \ref{Morse:end-points} and the fact that the projection maps $\{\pi_U: U \in \mf{S}\}$ are uniformly coarsely Lipschitz implies that
\[d_U(x_1,x_2) \leq d_U(x_1, \gamma) +\diam(\pi_U(\gamma)) +d_U(x_2,\gamma) \prec 2K_N+B',\]
completing the proof of the second coarse equality of Equation (\ref{eq:string_coarse_eq}).

To prove the final coarse equality in Equation (\ref{eq:string_coarse_eq}), let $T$ be a triangle in $\mc{X}$ with vertices $a,b,$ and $c$ such that $x_1$ is within $K_N$ of each side of $T$. Because $\mc{X}$ is a maximal almost HHS, $\pi_S(T)$ is an $(L,L)$--quasi-geodesic triangle for some $L$ determined by $N$. Because $\pi_S$ is coarsely Lipschitz, the distance between $\pi_S(x_1)$ and each side of  $\pi_S(T)$ is bounded above by a constant determined by $N$.
By Lemma \ref{Morse:end-points} and Remark \ref{remark:Morse_qausi-geoedsic}, the distance between each side of $\pi_S(T)$ and any $(1,20\delta)$--quasi-geodesics between its endpoints in $\partial CS$ is bounded by a constant determined by $N$ and $\delta$.  Thus, $\pi_S(x_1) \in \hat{m}^P_{CS}(a,b,c)$ for some $P$ depending only on $\delta$ and $N$. Similarly, $\pi_S(x_2) \in \hat{m}^P_{CS}(a,b,c)$. 
The last coarse equality now follows from Lemma \ref{lem:bounded_modified_centers} and Lemma \ref{lem:paulin2} part (1). 
\end{proof} 

Proposition \ref{small:cross-ratio} is now proved by combining Proposition \ref{relate:cross-ratios} and Lemma \ref{lem:paulin2} part (2). Thus we have the following corollary to Theorems \ref{thrm:main_body_part1} and \ref{thrm:main_body_part2}. 

\begin{corollary} \label{cor:main_body_HHS} 
Let $\mc{X}$ and $\mc{Y}$ be proper geodesic Morse center spaces that can be equipped with maximal almost HHS structures. Then a homeomorphism $\partial \mc{X} \rightarrow \partial \mc{Y}$ is induced by a quasi-isometry if and only if it and its inverse are 2--stable and quasi-m\"obius.
\end{corollary}

 Since every HHG with at least three points in its Morse boundary satisfies the hypotheses of Corollary \ref{cor:main_body_HHS}, we recover the Main Theorem from the introduction as a special case of Corollary \ref{cor:main_body_HHS}.  More broadly, Theorems \ref{thrm:main_body_part1} and \ref{thrm:main_body_part2} apply to any maximal almost HHS that has at least three points in its Morse boundary and admits a cocompact group action. 

 We conclude by noting a direct application of our results. If $(\mc{X},\mf{S})$ is an HHS and $\mc{Y}$ is a metric space  with a quasi-isometry $f\colon \mc{Y} \rightarrow \mc{X}$, then we can put an HHS structure on $\mc{Y}$ using the index set $\mf{S}$, hyperbolic spaces $\{CU\}_{U\in\mf{S}}$, and projection maps $\{\pi_U \circ f\}_{U\in \mf{S}}$. Thus Theorem \ref{thrm:main_body_part1} and Proposition \ref{small:cross-ratio} provided the following.

\begin{corollary}
Let $\mc{X}$ and $\mc{Y}$ be proper geodesic Morse center spaces with the small cross-ratio property. Suppose there exists a homeomorphism $h:\partial \mc{X} \rightarrow \partial \mc{Y}$ such that $h$ and $h^{-1}$ are quasi-m\"obuis and 2--stable. If $\mc{X}$ admits an HHS structure, then $\mc{Y}$ also admits an HHS structure.\end{corollary}

\bibliography{Morse}{}
\bibliographystyle{alpha}

\end{document}

%% file: paulin_cross_ratio2.pdf_tex
\begingroup%
  \makeatletter%
  \providecommand\color[2][]{%
    \errmessage{(Inkscape) Color is used for the text in Inkscape, but the package 'color.sty' is not loaded}%
    \renewcommand\color[2][]{}%
  }%
  \providecommand\transparent[1]{%
    \errmessage{(Inkscape) Transparency is used (non-zero) for the text in Inkscape, but the package 'transparent.sty' is not loaded}%
    \renewcommand\transparent[1]{}%
  }%
  \providecommand\rotatebox[2]{#2}%
  \ifx\svgwidth\undefined%
    \setlength{\unitlength}{502.02773261bp}%
    \ifx\svgscale\undefined%
      \relax%
    \else%
      \setlength{\unitlength}{\unitlength * \real{\svgscale}}%
    \fi%
  \else%
    \setlength{\unitlength}{\svgwidth}%
  \fi%
  \global\let\svgwidth\undefined%
  \global\let\svgscale\undefined%
  \makeatother%
  \begin{picture}(1,1.00396111)%
    \put(0,0){\includegraphics[width=\unitlength,page=1]{paulin_cross_ratio2.pdf}}%
    \put(0.44630443,1.01906718){\color[rgb]{0,0,0}\makebox(0,0)[lt]{\begin{minipage}{0.22081878\unitlength}\raggedright $a$\end{minipage}}}%
    \put(0.02014697,0.21683487){\color[rgb]{0,0,0}\makebox(0,0)[lt]{\begin{minipage}{0.12975949\unitlength}\raggedright $b$\end{minipage}}}%
    \put(0.87747016,0.18724063){\color[rgb]{0,0,0}\makebox(0,0)[lt]{\begin{minipage}{0.19805393\unitlength}\raggedright $c$\end{minipage}}}%
    \put(0.13442637,0.66120418){\color[rgb]{0,0,0}\makebox(0,0)[lt]{\begin{minipage}{0.3096016\unitlength}\raggedright $i_{ab}$\end{minipage}}}%
    \put(0.71948226,0.64435824){\color[rgb]{0,0,0}\makebox(0,0)[lt]{\begin{minipage}{0.45301992\unitlength}\raggedright $i_{ac}$\end{minipage}}}%
    \put(0.25735642,0.01149622){\color[rgb]{0,0,0}\makebox(0,0)[lt]{\begin{minipage}{0.45984943\unitlength}\raggedright \end{minipage}}}%
    \put(0.41215719,0.13988974){\color[rgb]{0,0,0}\makebox(0,0)[lt]{\begin{minipage}{0.377896\unitlength}\raggedright $i_{bc}$\end{minipage}}}%
    \put(0,0){\includegraphics[width=\unitlength,page=2]{paulin_cross_ratio2.pdf}}%
    \put(0.47474675,0.73662653){\color[rgb]{0,0,0}\makebox(0,0)[lt]{\begin{minipage}{0.27043228\unitlength}\raggedright $a_n$\end{minipage}}}%
    \put(0.18985936,0.30486678){\color[rgb]{0,0,0}\makebox(0,0)[lt]{\begin{minipage}{0.29618774\unitlength}\raggedright $b_n$\end{minipage}}}%
    \put(0.70429706,0.2760213){\color[rgb]{0,0,0}\makebox(0,0)[lt]{\begin{minipage}{0.39116093\unitlength}\raggedright $c_n$\\ \end{minipage}}}%
    \put(0.038308,0.43977547){\color[rgb]{0,0,0}\makebox(0,0)[lt]{\begin{minipage}{0.27043228\unitlength}\raggedright $b_n'$\end{minipage}}}%
    \put(0.27658988,0.84392746){\color[rgb]{0,0,0}\makebox(0,0)[lt]{\begin{minipage}{0.33965005\unitlength}\raggedright $a_n'$\end{minipage}}}%
    \put(0,0){\includegraphics[width=\unitlength,page=3]{paulin_cross_ratio2.pdf}}%
  \end{picture}%
\endgroup%